\documentclass[oneside,12pt]{article}

\usepackage{amsmath}
\usepackage{amssymb}
\usepackage{pxfonts}
\usepackage{graphicx}
\usepackage{eucal}
\usepackage{mathrsfs}
\usepackage{theorem}
\usepackage{pifont}
\usepackage{sectsty}
\usepackage{amscd}
\usepackage{color}
\usepackage{fancyhdr}
\usepackage{framed}
\usepackage[all]{xy}
\usepackage{hyperref}
\usepackage{dsfont}

\definecolor{shadecolor}{rgb}{0.8,0.8,0.8}

\newcommand{\deriv}[2]{\frac{d#1}{d#2}}
\newcommand{\pd}[2]{\frac{\partial#1}{\partial#2}}
\newcommand{\SO}[1]{\operatorname{SO}(#1)}
\renewcommand{\O}[1]{\operatorname{O}(#1)}
\newcommand{\dist}{\operatorname{dist}}
\newcommand{\IP}[2]{\left\langle #1,#2 \right\rangle}
\newcommand{\id}{\operatorname{Id}}
\newcommand{\Vol}{\text{\textup{Vol}}}
\newcommand{\Volg}{\text{\textup{Vol}}_\g}
\newcommand{\Volh}{\text{\textup{Vol}}_\h}
\newcommand{\g}{\mathfrak{g}}

\newcommand{\euc}{\mathfrak{e}}
\newcommand{\h}{\mathfrak{h}}
\newcommand{\vp}{\varphi}
\newcommand{\e}{\varepsilon}
\newcommand{\Ga}{\Gamma}
\newcommand{\Lam}{\Lambda}

\newcommand{\TM}{T\M}
\newcommand{\TN}{T\N}

\newcommand{\GLp}{\operatorname{GL}_d^+}
\newcommand{\GLm}{\operatorname{GL}_d^-}
\newcommand{\SOd}{\operatorname{SO}(d)}

\newcommand{\dVolg}{\text{\textup{dVol}}_\g}
\newcommand{\dVol}{\text{\textup{dVol}}}
\newcommand{\dVoln}{\text{\textup{dVol}}_{\g_n}}

\newcommand{\dVolh}{\text{\textup{dVol}}_{\mathfrak{h}}}

\newcommand{\weakly}{\rightharpoonup}

\newcommand{\Cof}{\operatorname{Cof}}
\newcommand{\cof}{\operatorname{cof}}
\newcommand{\Det}{\operatorname{Det}}
\newcommand{\tr}{\operatorname{tr}}
\newcommand{\Hom}{\operatorname{Hom}}
\newcommand{\image}{\operatorname{Image}}

\newcommand{\TstarM}{T^*\M}

\renewcommand{\div}{\operatorname{div}}

\newcommand{\M}{{\mathcal{M}}}
\newcommand{\N}{\mathcal{N}}
\newcommand{\W}{\Omega}
\newcommand{\R}{\mathbb{R}}
\newcommand{\calD}{\mathcal{D}}
\newcommand{\bbN}{\mathbb{N}}

\newcommand{\innerp}[1]{\left\langle#1\right\rangle}

\newcommand{\pl}{\partial}

\newcommand{\teq}[1]{\stackrel{\mathrm{#1}}{=}}

\newcommand{\inner}[1]{\left\langle#1\right\rangle}      % \inner{.} => <.>

\newcommand{\brk}[1]{\left(#1\right)}          % \brk{.}     => (.)
\newcommand{\Brk}[1]{\left[#1\right]}          % \Brk{.}     => [.]
\newcommand{\BRK}[1]{\left\{#1\right\}}        % \BRK{.}     => {.}
\newcommand{\Cases}[1]{\begin{cases} #1 \end{cases}}

\newcommand{\beq}{\begin{equation}}
\newcommand{\eeq}{\end{equation}}

\newcommand{\thmref}[1]{Theorem~\ref{#1}}
\newcommand{\defref}[1]{Definition~\ref{#1}}

\newcommand{\propref}[1]{Proposition~\ref{#1}}
\newcommand{\lemref}[1]{Lemma~\ref{#1}}
\newcommand{\corrref}[1]{Corollary~\ref{#1}}

\newcommand{\textand}{\quad\text{ and }\quad}
\newcommand{\Textand}{\qquad\text{ and }\qquad}
\newcommand{\sgn}{{\operatorname{sgn}}}

\theoremheaderfont{\fontfamily{pzc}\bfseries\large}
\newtheorem{theorem}{Theorem}[section]
\newtheorem{lemma}[theorem]{Lemma}
\newtheorem{proposition}[theorem]{Proposition}

\newtheorem{corollary}[theorem]{Corollary}
\newtheorem{definition}[theorem]{Definition}

{\theorembodyfont{\rmfamily}

\newenvironment{proof}{{\flushleft \emph{Proof}:}}{\hfill\ding{110}}

\newenvironment{remark}{{\flushleft \fontfamily{pzc}\bfseries\large Remark:}}{}

\newenvironment{example}{{\flushleft {\fontfamily{pzc}\bfseries\large Example}:}}{\hfill $\blacktriangle\blacktriangle\blacktriangle$}

\makeatletter
\newcommand{\extp}{\@ifnextchar^\@extp{\@extp^{\,}}}
\def\@extp^#1{\mathop{\bigwedge\nolimits^{\!#1}}}
\makeatother

\numberwithin{equation}{section}

\begin{document}

\title{Reshetnyak rigidity for Riemannian manifolds}
\author{Raz Kupferman\footnote{Institute of Mathematics, The Hebrew University.
}, Cy Maor\footnote{Department of Mathematics, University of Toronto.} $\,$ and Asaf Shachar\footnotemark[1]}
\date{}
\maketitle

%%%%%%%%%%%%%%%%%%%%%%%%%
\begin{abstract}
We prove two rigidity theorems for maps between Riemannian manifolds.
First, we prove that a Lipschitz map $f:\M\to\N$ between two oriented Riemannian manifolds, whose differential is almost everywhere an orientation-preserving isometry, is an isometric immersion.
This theorem was previously proved using regularity theory for conformal maps; we give a new, simple proof, by generalizing the Piola identity for the cofactor operator.
Second, we prove that if there exists a sequence of mapping $f_n:\M\to\N$, whose differentials converge in $L^p$ to the set of orientation-preserving isometries, then there exists a subsequence converging to an isometric immersion.
These results are generalizations of celebrated rigidity theorems by Liouville (1850) and Reshetnyak (1967) from Euclidean to Riemannian settings. 
Finally, we describe applications of these theorems to non-Euclidean elasticity and to convergence notions of manifolds.
\end{abstract}

\tableofcontents

%%%%%%%%%%%%%%%%%%%%%%%%%
\section{Introduction, main results and applications}

In 1850, Liouville proved a celebrated rigidity theorem for conformal mappings \cite{Lio50}. An important corollary of Liouville's theorem is that a sufficiently smooth mapping $f:\Omega\subset \R^d\to \R^d$ that is everywhere a local isometry must be a global isometry. Specifically, if $f\in C^{1}(\Omega;\R^d)$ satisfies $df\in \SO{d}$ everywhere, then $f$ is an affine function, i.e., an isometric embedding of $\Omega$ into $\R^d$.

While from a modern perspective it seems rather trivial, Liouville's rigidity theorem was generalized in various highly non-trivial directions. 
One such direction is concerned with the regularity requirements on $f$.
As it turns out, it suffices to require that $f$ be Lipschitz continuous with $df\in \SO{d}$ almost everywhere (by Rademacher's theorem, Lipschitz continuous functions are a.e.~differentiable).
Indeed, any differentiable map $f$ satisfies (a weak form of)
\beq
\label{eq:Piola_id_Euclidean}
\div \cof df = 0,
\eeq
where $\cof df$ is the cofactor matrix, and the divergence operates row-wise \cite[Chapter 8.1.4.b.]{Eva98};
in the context of elasticity theory, identity \eqref{eq:Piola_id_Euclidean} is known as the \emph{Piola identity} \cite[Section 1.7]{Cia88}.
If $df\in \SO{d}$, then $df = \cof df$.
The Piola identity \eqref{eq:Piola_id_Euclidean} then implies that $f$ is weakly-harmonic, hence by Weyl's lemma, $f$ is smooth \cite{Res67b};
for a more complete survey on regularity see \cite{Lor13}.

Another type of generalization is due to Reshetnyak \cite{Res67}. It is concerned with sequences of mappings $f_n:\Omega\to\R^d$ that are asymptotically locally rigid in an average sense. Specifically,

%%%%%%%%
\begin{quote}
\emph{
Let $\Omega\subset \R^d$ be an open, connected, bounded domain, and let $1\le p<\infty$.
If $f_n\in W^{1,p}(\Omega;\R^d)$ satisfy $\int_\Omega f_n\,dx = 0$ and
\[
\lim_{n\to \infty} \int_\Omega \dist^p(df_n, \SO{d})\,dx = 0,
\]
then $f_n$ has a subsequence converging in the strong $W^{1,p}(\Omega;\R^d)$ topology to an affine mapping.
}
\end{quote}
%%%%%%%

Here, $\dist(df, \SO{d}):\Omega\to\R$ is a measure of local distortion of $f$. Liouville's theorem (for Lipschitz mappings) states that if this local distortion vanishes almost everywhere, then $f$ is an isometric embedding. Reshetnyak's theorem states that a sequence of mappings for which the $L^p$-norm of the local distortion tends to zero, converges (modulo a subsequence) to an isometric embedding.
There exist many other generalizations of these rigidity theorems, for conformal mappings, multi-well potentials and so on, but they are farther away from the context of this paper.

%%%%%%%%%%%%%%%%%%%%%%%%%%%%%%%%%%%%%%%%%%%%%%%%%%%
\subsection{Liouville's theorem and the Piola identity for Riemannian manifolds}
\label{sec:int_Liouville_thm}

This paper is concerned with generalizations of Liouville's and Reshetnyak's theorems for mappings between Riemannian manifolds. 
In this sense, it deals with a rigidity of Riemannian manifolds. 
We note that in the literature, the term "rigidity of manifolds" may refers to many other things, e.g., questions of boundary rigidity and inverse problems on manifolds (see for example the survey \cite{Cro04}), or rigidity of submanifolds \cite[Chapter 12]{Spi99V};
the rigidity results presented in this paper are of a different nature.

Throughout this paper, $(\M,\g)$ and $(\N,\h)$ are compact, connected, oriented $d$-dimensional Riemannian manifolds (possibly with a $C^1$ boundary). 
The role of $\SO{d}$ is now played by $\SO{\g_x,\h_y}$---the set of orientation preserving transformations $T_x\M\to T_y\N$ (which by a choice of positively-oriented orthonormal frames, can be identified with $\SO{d}$).
Liouville's theorem for smooth mappings has a well-known generalization for manifolds:

%%%%%%%
\begin{quote}
\emph{
Let $f\in C^1(\M;\N)$ satisfy $df\in\SO{\g,f^*\h}$ everywhere in $\M$. Then, $f$ is smooth and rigid, in the sense that every $x\in\M$ has a neighborhood $U_x$ in which
\[
f = \exp^\N_{f(x)} \circ df_x\circ (\exp^\M_x)^{-1}.
\]
}
\end{quote}
%%%%%%% 

Here, $\exp^M$ and $\exp^\N$ are the respective exponential maps in $\M$ and $\N$; for $x\in\M$, $\SO{\g,f^*\h}_x=\SO{\g_x,\h_{f(x)}}$.
The generalization of Liouville's theorem for smooth mappings states that a Riemannian isometry can be (locally) factorized via the mapping $f(x)$ and its derivative $df_x$ at a single point.

A first natural question is whether this generalization of Liouville's theorem holds if $f$ is assumed less regular. Namely:

%%%%%%%%%
\begin{theorem}[Liouville's rigidity for Lipschitz functions]
\label{thm:Liouville_M_N_Lipschitz}
Let $f\in W^{1,\infty}(\M;\N)$ satisfy $df\in \SO{\g,f^*\h}$ almost everywhere.
Then $f$ is smooth, hence a smooth isometric immersion.
\end{theorem}
%%%%%%%%%

This theorem was proved, in the wider context of the regularity of conformal maps \cite{Res78,Res94,LS14} (see also \cite{Har58,CH70,Tay06} for other results on the regularity of isometries).
The techniques used in these references are rather different from the simple argument based on the Piola identity \eqref{eq:Piola_id_Euclidean} and harmonicity in Euclidean space. 
Our first result is a new and simple proof to Theorem~\ref{thm:Liouville_M_N_Lipschitz}, that builds upon those very same arguments.
We first prove a Riemannian version of the Piola identity (Proposition~\ref{prop:divCof=0}):
\beq
\label{eq:Piola_id_general_strong}
\delta_\nabla \Cof d\vp =0
\eeq
valid for every $\vp\in C^2(\M;\N)$, where $\delta_\nabla$ is the co-differential induced by the Riemannian connection on $\vp^*T\N$, and $\Cof$ is an intrinsic cofactor operator (see Section~\ref{sec:det_cof}, and Section~\ref{sec:Piola_id_coor} for an expression in local coordinates).
We then prove a weak version of \eqref{eq:Piola_id_general_strong}, by embedding $\N$ isometrically into a Euclidean space of higher dimension:

\begin{theorem}
\label{thm:weak_Piola_identity}[Piola identity, weak formulation]
Let $f\in W^{1,p}(\M;\N)$ where $p\ge 2(d-1)$ ($p>2$ if $d=2$). Let $\iota:\N\to\R^D$
be an isometric embedding of $\N$ in $\R^D$ with second fundamental form $A$.
Then, for every $\xi \in W_0^{1,2}(\M;\R^D)\cap L^\infty(\M;\R^D)$,
\beq
\int_\M\innerp{f^*d\iota\circ\Cof df, \nabla \xi}_{\g,\euc} \,\dVolg =
\int_\M \innerp{\tr_\g f^*A(\Cof df,df),\xi}_\euc \,\dVolg,
\label{eq:Cof_df_orthogonal_extrinsic_a}
\eeq
where $\euc$ is the Euclidean metric on $\R^D$. $\nabla \xi$ is the trivial connection $\nabla^{\M\times\R^D}$ on the bundle $\M\times\R^D$. In other words, it is just a componentwise-differentiation.
\end{theorem}

Since $df\in \SO{\g,f^*\h}$ implies $\Cof df = df$ (Corollary \ref{cor:DetAndCof} below), we obtain the following corollary, from which Theorem~\ref{thm:Liouville_M_N_Lipschitz} follows immediately:

%%%%%%%%%%
\begin{corollary}[A.e.~local isometries are harmonic]
\label{cor:Liouville_M_N_Lipschitz}
Let $f\in W^{1,\infty}(\M;\N)$ satisfy $df\in \SO{\g,f^*\h}$ almost everywhere.
Then, $f$ is weakly-harmonic in the sense of \cite{Hel02}:
\beq
\int_\M\innerp{d(\iota\circ f), \nabla \xi}_{\g,\euc} \,\dVolg =
\int_\M \innerp{\tr_\g f^*A( df,df),\xi}_\euc \,\dVolg,
\eeq
for all $\xi \in W_0^{1,2}(\M;\R^D)\cap L^\infty(\M;\R^D)$.
In particular, by the regularity theorem for continuous, weakly-harmonic mappings \cite[Theorem 1.5.1]{Hel02}, $f$ is smooth.
\end{corollary}
%%%%%%%%%

The combination of intrinsic and extrinsic approaches is essential to our approach: on the one hand, the Piola identity cannot be formulated for mappings $\M^d\to\R^D$, $d<D$, at least in a way by which the harmonicity of $f$ can be deduced.
On the other hand, it is not clear how to formulate a weak form of this identity without an isometric embedding into a Euclidean space, which naturally embeds $\TstarM\otimes f^*T\N$ into a (smooth) vector bundle $\TstarM\otimes\R^D$ independent of $f$.

We note that the Piola identity is of importance beyond the present context, as a fundamental identity in elasticity theory, see e.g.~\cite[Section 1.7]{Cia88} and \cite[Chapter~1.7]{MH83}.
A generalization of the Piola identity to manifolds appears in \cite[Chapter~1, Theorem~7.20]{MH83}, however, its formulation is slightly different from ours; it is not stated in the language of vector-valued forms and their exterior derivatives, which is the formulation needed here. Also, it lacks a weak version, and its coordinate formulation is wrong (a correct one is given in Section~\ref{sec:Piola_id_coor}).

%%%%%%%%%%%%%%%%%%%%%%%%%%%%%%%%%%%%%%%%%%%%%%%
\subsection{Reshetnyak's theorem for Riemannian manifolds}

A second natural question is whether a generalization of Reshetnyak's theorem can be established for mappings between manifolds. Suppose that $\M$ can be mapped into $\N$ with arbitrarily small mean local distortion. Can one deduce that $\M$ is isometrically immersible in $\N$? Moreover, suppose that those mappings are diffeomorphisms. Can one deduce that $\M$ and $\N$ are isometric? 

Our main result is a generalization of Reshetnyak's theorem, which answers positively all of these questions:
\begin{theorem}
\label{thm:Reshetnyak_M_N}
Let $(\M,\g)$ and $(\N,\h)$ be compact, oriented, $d$-dimensional Riemannian manifolds with $C^1$ boundary. Let $1\le p<\infty$ and let $f_n\in W^{1,p}(\M;\N)$ be a sequence of mappings satisfying  
\beq
\dist(d f_n, \SO{\g,f_n^*\h})\to 0 \qquad\text{in $L^p(\M)$}.
\label{eq:Lpconv_of_dist}
\eeq
Then, $\M$ can be immersed isometrically into $\N$, and there exists a subsequence of $f_n$ converging in $W^{1,p}(\M;\N)$ to a smooth isometric immersion $f:\M\to\N$.

Moreover, if $f_n(\pl \M)\subset \pl \N$ and $\Volg\M=\Volh\N$, then $\M$ and $\N$ are isometric and $f$ is an isometry.
In particular, these conditions hold if $f_n$ are diffeomorphisms.
\end{theorem}
%%%%%%%%%%

Before giving a sketch of the proof, we explain how we measure the distance of $df$ from $\SO{\g,f^*\h}$ for a given mapping $f:\M \to \N$.
Recall that for $x\in\M$, $\SO{\g,f^*\h}_x=\SO{\g_x,\h_{f(x)}}$; the Riemannian metrics on $\M,\N$ induce an inner-product on $T_x^*\M\otimes T_{f(x)}\N \simeq \text{Hom}(T_x\M,T_{f(x)}\N)$. We measure the distance $df_x$ from $\SO{\g,f^*\h}_x$ using the distance induced by this inner product.

Fixing orthonormal frames in $T_x\M$ and $T_{f(x)}\N$, this reduces to the standard Euclidean distance from $\SO{d}$; 
an expression in local coordinates can be obtained as follows:
fix local coordinates at $\M$ at $x$ and at $\N$ at $f(x)$.
Denote by $\g,\h$ the coordinate representations of the Riemannian metrics and by $\sqrt{\g},\sqrt{h}$ their unique symmetric positive-definite square roots.
Then, $\sqrt{\g} \in \SO{\g,\euc}$ and $\sqrt{\h} \in \SO{\h,\euc}$.
A straightforward calculation shows that,
\[
\dist(df,\SO{\g,f^*\h})|_x = \dist\brk{\sqrt{\h(f(x))}\circ d f(x) \circ \sqrt{\g^{-1}(x)} ,\SO{d}},
\]
where on the right-hand side we use the standard Euclidean distance between matrices.
Note that this expression is valid at $x$; it can be extended to a neighborhood of $x$ only if $f$ is "localizable", see~\cite{LS14} (for example if $f$ is continuous), not necessarily for every $f\in W^{1,p}(\M;\N)$.

%%%%%%%
\paragraph{Sketch of proof}
We present a rough sketch of the proof, emphasizing its main ideas; applications of the theorem are discussed farther below.

As a starting point, note the following well-known linear algebraic fact: 
$A\in \SO{d}$ if and only if $\det A = 1$ and $\cof A=A$, where $\cof A$ is the cofactor matrix of $A$, i.e., the transpose of its adjugate.
This fact can be reformulated for mappings between manifolds: 
$df\in \SO{\g,f^*\h}$ if and only if $\Det df=1$ and $\Cof df =df$, where $\Det$ and $\Cof$ are intrinsic determinant and cofactor operators (see Section~\ref{sec:det_cof} for details).

The assumptions on the sequence $(f_n)$ imply that it is precompact in the weak $W^{1,p}$-topology.
However, the direct method of the calculus of variations cannot be used to deduce that a limit function $f$ is an isometric immersion, since the functional \eqref{eq:the_functional} is not lower-semicontinuous with respect to the weak $W^{1,p}$-topology.
Instead, we follow the ideas behind the proof of \cite{JK90} to Reshetnyak's (Euclidean) rigidity theorem; 
we use Young measures 
to show that any weak limit $f$ of $(f_n)$ must satisfy $\Det df=1$ and $\Cof df =df$ a.e., hence $df\in \SO{\g,f^*\h}$ a.e.

The generalization of \cite{JK90} is not straightforward. The fundamental theorem of Young measures applies to sequences of vector-valued functions. A generalization of this theory to sections of a \emph{fixed} vector bundle is relatively straightforward (see Section~\ref{sec:Young_meas} for details). In our case, however,  $df_n$ is a section of $\TstarM\otimes f_n^*T\N$, i.e., every $df_n$ is a section of a \emph{different} vector bundle. Trying to overcome this difficulty by the standard procedure of embedding $\N$ isometrically into a  high-dimensional Euclidean space $\R^D$ (so that all $df_n$ become sections of the same vector bundle $\TstarM\otimes \R^D$) does not solve the problem, because information about orientation is lost (as discussed below,  \thmref{thm:Reshetnyak_M_N} does not hold if $\SO{\g,f_n^*\h}$ is replaced by $\O{\g,f_n^*\h}$). This difficulty is overcome by a combination of extrinsic (embedded) and intrinsic (local) treatments of $\N$ in different parts of the argument.

In addition, this generalization of \cite{JK90} only works for $p>d$, otherwise $\Det df_n\not\weakly \Det df$, and even worse, the use of local coordinates in the intrinsic analysis is impossible. 
To encompass the case $1\le p\le d$, we use a truncation argument from \cite{FJM02b,LP10}, adapted to our setting.

Having obtained that $df\in \SO{\g,f^*\h}$ a.e., we use Theorem~\ref{thm:Liouville_M_N_Lipschitz} to obtain that $f$ is a smooth isometric immersion.
The proof of Theorem~\ref{thm:Reshetnyak_M_N} is completed by showing that $f_n\to f$ in the strong (rather than weak) $W^{1,p}(\M;\N)$ topology  (using Young measures again). 
This stronger convergence, along with the conditions that $f_n(\pl \M)\subset \pl \N$ and $\Volg\M=\Volh\N$, imply that $f$ is an isometry. 
Note that this last part has no equivalence in the Euclidean version of Reshetnyak's theorem.

%%%%%%%%%%%%%%%%%
\subsection{Applications}

We present two applications of Theorem~\ref{thm:Reshetnyak_M_N}.
The first is in the field of \emph{non-Euclidean elasticity} (also known as \emph{incompatible elasticity}),
which is a branch of non-linear elasticity concerned with elastic bodies that do not have a reference configuration, i.e., a stress-free  state.
Such bodies are typically modeled as Riemannian manifolds $(\M,\g)$, and the ambient space is another manifold $(\N,\h)$ of the same dimension.
A body does not have a reference configuration if $(\M,\g)$ cannot be isometrically embedded in $(\N,\h)$.
For example, if $\g$ is non-flat and $\N=\R^d$, $\M$ does not have a reference configuration.
Such ``intrinsically curved" elastic bodies are very common in many physical and biological models, usually due to material defects or inhomogeneous shrinkage or growth; these change the equilibrium distances between adjacent material points, resulting in an intrinsic non-Euclidean geometry. 
The ambient space may be curved if the elastic body is constrained to some curved space, or in general relativistic applications.
Recent examples for non-Euclidean elastic problems in the physics literature can be found in \cite{KES07,SRS07,ESK09,OY09,KVS11,DCGRLL13,ESK13,Efr15,AKMMS16}, and in the mathematical literature in \cite{LP10,KS12,LRR,BLS16,ALL17,KOS17,Olb17,KO18} (this is by no means a comprehensive list).

The elastic energy associated with a configuration $f:\M\to \N$, 
\beq
\label{eq:elastic_energy}
E_{\M,\N}(f) = \int_\M W(df)\, \dVolg
\eeq
is model-dependent, however it typically admits a lower bound
\beq
\label{eq:energy_bound}
W(df) \ge C\dist^p(df,\SO{\g,f^*\h})
\eeq
for some exponent $p\ge 1$ (usually $p=2$), and $C>0$.

A natural question in this context is whether a ``geometric incompatibility"---that the body manifold cannot be isometrically immersed in the space manifold (i.e., the lack of a reference configuration)---is equivalent to  an ``energetic incompatibility"---that the elastic energy is bounded away from zero.
An immediate corollary of Theorem~\ref{thm:Reshetnyak_M_N} answers this affirmatively:

\begin{corollary}
\label{cr:non_trivial_inf_energy_in_the_elastic_functional}
Let $\M$ be a compact $d$-dimensional manifold with boundary, and let $\N$ be either $\R^d$, or a compact $d$-dimensional manifold with boundary.  If $\M$ is not isometrically immersible in $\N$, then
\[
\inf_{f\in W^{1,p}(\M;\N)} E_{\M,\N}(f) > 0
\]
whenever $E_{\M,\N}$ satisfies \eqref{eq:energy_bound} for some $p\ge 1$.
\end{corollary}
For $\N=\R^d$, this result was obtained in \cite[Theorem 2.2]{LP10} using different methods.

A second application of Theorem~\ref{thm:Reshetnyak_M_N} is concerned with notions of convergence of Riemannian manifolds that arise in the study of homogenization of defects \cite{KM15,KM15b} and in structural stability of non-Euclidean elasticity \cite{KM16}.
A sequence of Riemannian manifolds $(\M_n,\g_n)$ converges to a Riemannian manifold $(\M,\g)$
if (up to some additional assumptions) there exist diffeomorphisms $F_n:\M\to \M_n$, such that 
\[
\dist(dF_n,\SO{\g,F_n^*\g_n})\to 0 \qquad\text{in $L^p(\M)$},
\]
and similarly for $F_n^{-1}$ (in other words, if the infimum elastic energy between $\M_n$ and $\M$ vanishes asymptotically).
However, in these works additional assumptions were made in order to guarantee the well-definiteness of the limit, that is, its independence on the choice of $F_n$.
In Section~\ref{sec:applications} we show that Theorem~\ref{thm:Reshetnyak_M_N} implies that such a notion of convergence is well-defined for $p$ large enough, without any further assumptions. 
Moreover, we present some examples, showing that this notion of metric convergence can be substantially different from Gromov--Hausdorff convergence.

%%%%%
\paragraph{The role of orientation}

Liouville's theorem for smooth mappings holds if $\SO{\g,f^*\h}$ is replaced with $\O{\g,f^*\h}$: indeed, a $C^1(\M;\N)$ mapping, which is everywhere a local isometry, is either globally orientation-preserving or globally orientation-reversing, which is equivalent in either case to the differential being in $\SO{\g,f^*\h}$.
However, both for Lipschitz mappings, and for asymptotically-rigid mappings, Liouville's and Reshtnyak's theorems do not hold if $\SO{\g,f^*\h}$ is replaced with $\O{\g,f^*\h}$  (even in a Euclidean setting, as exemplified by the map $x\mapsto |x|$ on the real line).

The reason for the breakdown of both rigidity theorems is the following:
maps whose differentials switch between the two connected components of $\O{\g,f^*\h}$ can be highly irregular, since $\O{\g,f^*\h}$ is rank-one connected. 
For example, it was proved in \cite{Gro86}, using methods of convex integration, that given an arbitrary metric $\g$ on the $d$-dimensional closed disc $\calD^d$, there exists a mapping $f\in W^{1,\infty}(\calD^d,\R^d)$, such that $f^\star \euc = \g$ a.e. (i.e., $df\in \O{\g,\euc}$ a.e.); see also \cite[Remark 2.1]{LP10}.
It follows that a functional such as
\beq
f\mapsto \int_\M |\g - f^*\h|^p\, \dVolg,
\label{eq:not_the_functional}
\eeq
which does not account for orientation, is not a good measure of distortion, even though at first sight, it might seem more natural than 
\beq
f\mapsto \int_\M \dist^p(d f, \SO{\g,f^*\h}) \, \dVolg.
\label{eq:the_functional}
\eeq
These difficulties only arise when mappings can switch orientations;  
the results of this paper hold if $\SO{\g,f^*\h}$ in \eqref{eq:the_functional} is replaced with $\O{\g,f^*\h}$ or with \eqref{eq:not_the_functional}, but the mappings are restricted to (local) diffeomorphisms.

%%%%
\paragraph{Open questions}
\begin{enumerate}
\item
A discussion of generalizations of Liouville's rigidity theorem cannot be complete without mentioning the far-reaching result of \cite{FJM02b}, which is a quantitative version of Reshetnyak's theorem:
\begin{quote}
\emph{
Let $\Omega\subset \R^d$ be an open, connected Lipschitz domain, and let $1< p<\infty$.
Then, there exists a constant $C>0$ such that for every $f\in W^{1,p}(\Omega;\R^d)$ there exists an affine map $\tilde f$ such that
\[
\| f -\tilde f\|_{W^{1,p}(\Omega;\R^d)}^p \le C \int_\Omega \dist^p(df, \SO{d})\,dx.
\]
}
\end{quote}
This theorem has been generalized in various ways, see e.g.~\cite{Lor16,CM16} and the references therein.
All these generalizations are in Euclidean settings.
A natural question is whether this theorem can be generalized to mappings between Riemannian manifolds. 

\item
While the results of this paper imply that for $\M$ not isometrically immersible into $\N$,
\[
\inf_{f\in W^{1,p}(\M;\N)} \int_\M \dist_{(\g,f^*\h)}^p(df,\SO{\g,f^*\h})\,\dVolg > 0,
\]
they do not provide an estimate on how large this infimum is.
Since the local obstruction to isometric immersibility can be related to a mismatch of curvatures, one would expect curvature-dependent lower bounds. 
Some results in this direction exist for $\N=\R^d$ \cite{KS12}, and some asymptotics for thin manifolds appear in \cite{BLS16,LRR,MS18}, however, the general picture is still widely open. 

\item
In this paper, we assume for simplicity that the Riemannian metrics $\g$ and $\h$ are smooth;
all the results hold for metrics of class $C^{1,\alpha}$.
It is of interest whether our results can be extended to less regular metrics including singularities.
In the context of the convergence of manifolds presented in Section~\ref{sec:applications}, an important example is the convergence  of locally-flat surfaces with conic singularities  \cite{KM15b,KM16}. 
The uniqueness of the limit in such cases is yet to be established.
\end{enumerate}

%%%%
\paragraph{Structure of this paper} 
In Section~\ref{sec:harmonic_analysis} we define the intrinsic notions of determinant and cofactor used throughout the paper, and state their important properties (for completeness, some properties whose proofs are more difficult to find in this generality are proven in the appendix). 
We then prove the strong and weak formulations of the Piola identity.
In Section~\ref{sec:asymptotic_rigidity}, we prove the generalization of Reshetnyak's asymptotic rigidity theorem for mappings between Riemannian manifolds (Theorem~\ref{thm:Reshetnyak_M_N}).
In Section~\ref{sec:applications}, we present the above-mentioned application of Theorem~\ref{thm:Reshetnyak_M_N} to the  convergence of manifolds.

%%%%%%%%%%%%%%%%%%%%%%%%%%%%%%%%%%%%%%%%%%%%%%%%%%%

\section{The Piola identity for Riemannian manifolds}
\label{sec:harmonic_analysis}

In this section, we derive the strong and weak formulations of the Piola identity between general Riemannian manifolds (Proposition~\ref{prop:divCof=0} and Theorem~\ref{thm:weak_Piola_identity}).
We start by defining determinant and  cofactor operators between general oriented inner-product spaces (Section~\ref{sec:det_cof}).
Given these definitions, one can prove the strong Piola identity \eqref{eq:Piola_id_general_strong} by a direct calculation in local coordinates, or using vector-valued forms; both of these proofs involve lengthy calculations, and are not very illuminating.
Here we take a more conceptual route, showing that the Piola identity is in fact the Euler-Lagrange equation of a null-Lagrangian (Sections~\ref{sec:Null_Lag} and \ref{sec:Piola_id_strong}).
After deriving in Section~\ref{sec:Piola_id_weak} a weak form of the Piola identity, we formulate, for completeness, in Section~\ref{sec:Piola_id_coor} both forms of the Piola identity in local coordinates.

%%%%%%%%%%%%%%%%%%%%%%%%%%%%%%%%%%%%%%%%%%%%%%%%%%%
\subsection{Intrinsic determinant and cofactor}
\label{sec:det_cof}

%%%%%%%%%%
\begin{definition}[determinant]
\label{def:intrinsic_det}
Let $V$ and $W$ be $d$-dimensional, oriented, inner-product spaces. Let $\star_V^k:\Lambda_k(V)\to\Lambda_{d-k}(V)$ and $\star_W^k:\Lambda_k(W)\to\Lambda_{d-k}(W)$ be their respective Hodge-dual operators. Let $A\in\Hom(V,W)$.
The \textbf{determinant} of $A$, $\Det A\in \R$, is defined by
\[
\Det A := \star^d_W \circ \extp^d A \circ \star^0_V,
\]
where $\extp^d A = A \wedge \ldots \wedge A$, $d$ times, and we identify $\extp^0 V \simeq \extp^0 W \simeq \R$.
\end{definition}
%%%%%%%%

This definition of the determinant matches the definition of the determinant of the matrix representing $A$ with respect to any positively-oriented orthonormal bases of $V$ and $W$.
In particular, let  $(\M,\g)$ and $(\N,\h)$ be oriented $d$-dimensional Riemannian manifolds.
We denote by $\star_\M^k:\Lambda_k(T\M)\to\Lambda_{d-k}(T\M)$ and $\star_\N^k:\Lambda_k(T\N)\to\Lambda_{d-k}(T\N)$ the Hodge-dual operators of the tangent bundles (note that the Hodge-dual in Riemannian settings usually applies to the exterior algebra of the cotangent bundle).
Let $f:\M\to \N$ be a differentiable mapping. 
Then,
\[
\Det df = \star^d_\N \circ \extp^d df \circ \star^0_\M = \frac{f^\star \dVolh}{\dVolg}.
\]
The last equality, as some other properties of the determinant are detailed in Appendix~\ref{app:det_cof}.

%%%%%%%%%
\begin{definition}[cofactor operator]
\label{def:intrinsic_cof}
Let $V$ and $W$ be $d$-dimensional, oriented, inner-product spaces.
Let $A\in\Hom(V,W)$.
The cofactor  of $A$, $\Cof A \in\Hom(V,W)$, is defined by
\[
\Cof A := (-1)^{d-1} \star_W^{d-1} \circ \extp^{d-1} A \circ \star_V^1,
\]
where we identify $\extp^1 V \simeq V$ and $\extp^1 W \simeq W$.
\end{definition}
%%%%%%%%%

Properties of the cofactor are presented in Appendix~\ref{app:det_cof}.
In particular, we prove the following identity, which is an intrinsic version of well-known properties of the matrix-cofactor:
\[
\Det A \, \id_V = A^T \circ \Cof A = (\Cof A)^T \circ A.
\]

An immediate corollary is:

%%%%%%%%%
\begin{corollary}
\label{cor:detA_eq_1_cofA_eq_A}
Let $V$ and $W$ be $d$-dimensional, oriented, inner-product spaces.
Let $A\in\Hom(V,W)$.
Then $A\in\SO{V,W}$ if and only if $\Det A=1$ and $\Cof A = A$.
\end{corollary}
%%%%%%%%%

In the context of differentiable mappings, $f:\M\to\N$, Corollary~\ref{cor:detA_eq_1_cofA_eq_A} implies that

%%%%%%%%%
\begin{corollary}
\label{cor:DetAndCof}
Let $(\M,\g)$ and $(\N,\h)$ be oriented $d$-dimensional manifolds. Then,
\[
df\in \SO{\g,f^*\h}
\]
if and only if
\[
\Det df = 1
\textand
\Cof df = df.
\] 
\end{corollary}
%%%%%%%%

\begin{remark}
\begin{enumerate}
\item 
For $d>2$, it can easily be checked that $A\in\SO{V,W}$ if and only if $\Cof A = A\ne 0$ (the condition on the determinant is satisfied automatically).
For $d=2$, $\Cof:\operatorname{Hom}(V,W)\to \operatorname{Hom}(V,W)$ is a linear operator;
the set $\BRK{\Cof A = A}$
is a linear subspace, consisting of all weakly conformal maps, i.e.~the maps $\lambda R$ for $\lambda\ge 0$ and $R\in \SO{V,W}$.
Therefore, we can apply  Theorem~\ref{thm:weak_Piola_identity}  in two dimensions  to weakly conformal maps, rather than to isometries, resulting in an equivalent of Corollary~\ref{cor:Liouville_M_N_Lipschitz} for weakly conformal maps: 
Let $f\in W^{1,p}(\M;\N)$, $\dim \M=\dim \N=2$ and let $p>2$. If $df$ is a weakly conformal map a.e., then $f$ is weakly-harmonic, and in particular smooth.
This is a known result \cite[Section 2.2, Example 11]{HW08}, \cite{LS14}; we mention it here as another example of the usefulness of the Riemannian version of Piola's identity.

\item 
  The characterization of isometries through cofactors (\corrref{cor:detA_eq_1_cofA_eq_A}) and the role of dimension can be illuminated by the following simple heuristic:  
  $\Cof A$ defines the action of $A$ on $(d-1)$-dimensional parallelepipeds, i.e., it determines volume changes of $(d-1)$-dimensional shapes, whereas, $A$ determines volume changes of $1$-dimensional shapes (lengths). When $d-1\ne 1$, the condition $\Cof A = A$ implies that metric changes in two different dimensions are fully correlated, which is a rigidity constraint, forcing $A$ to be either trivial, or an isometry (cf. $r^{d-1}=r$ if and only if $r=0$ or $r=1$).
  \end{enumerate}
\end{remark}

%%%%%%%%%%%%%%%%%%%%%%%%%%%%%%%%%%%%%%%%%%%%%%%%%%%
\subsection{Null-Lagrangians} 
\label{sec:Null_Lag}
A functional $E$ is a \emph{null-Lagrangian} if every smooth map is a critical point of $E$, with respect to variations that do not alter the boundary.

\begin{lemma}
\label{lem:integrals_homotopic}
Let $\M,\N$ be smooth manifolds of dimensions $m,n$ respectively, $\M$ compact and oriented ($\M$ and $\N$ can both have boundaries). Let $\omega \in \Omega^m(\N)$ be closed. 
Let $f_0, f_1:\M\rightarrow \N$ be smooth maps which are homotopic relative to $\partial \M$.
Then
 \[
\int_{\M}f_0^*\omega=\int_{\M} f_1^*\omega.
 \]
\end{lemma}

\begin{proof}
Let $F:\M\times I\rightarrow \N$ be a smooth homotopy between $f_0$ and $f_1$ relative to $\pl \M$, i.e.~$f_t|_{\partial \M}=f_0|_{\partial \M}$ for every $t$. 
Since $d$ commutes with pullbacks, Stokes theorem implies that
\[
\begin{split}
 0 &= \int_{\M\times I} F^\ast d\omega = \int_{\M\times I} dF^\ast \omega= \int_{\partial(\M\times I)} F^\ast \omega \\ &= \int_{\M\times\{1\}} F^\ast \omega - \int_{\M\times\{0\}}F^\ast \omega + \int_{\partial \M \times (0,1)} F^\ast \omega \\ &= \int_{\M} f_1^\ast\omega - \int_{\M}f_0^\ast \omega,
 \end{split}
 \]
where $ \int_{\partial \M \times (0,1)} F^\ast \omega=0$ follows from the fact the homotopy respects the boundary, hence, the restriction of $F^\ast\omega$ to $\partial \M \times (0,1)$ is identically zero.
\end{proof}

\begin{corollary}[Pullbacks of closed forms are null Lagrangians]
\label{cor:Null-Lag1}
Let $\M,\N$ and $\omega$ be as in \lemref{lem:integrals_homotopic}. Let $E:C^{\infty}(\M,\N) \to \R$ be defined by $E(f)=\int_{\M} f^*\omega$.  Then $E$ is a null-Lagrangian. 
\end{corollary}

\begin{proof}
Let $f_t:\M \to \N$ be a smooth variation relative to $\pl \M$ of $f_0=f$. By \lemref{lem:integrals_homotopic}, 
$E(f_t)=E(f_0)$, so $E(f_t)$ is constant. 
\end{proof}

In the case where $n=m$, every $m$-form on $\N$, and in particular the volume form, is closed. Hence:
\begin{corollary}
\label{cor:Null-Lag2}
Let $\M$ and $\N$ be $d$-dimensional, smooth, oriented Riemannian manifolds, with $\M$ compact; Then, the functional $E:C^{\infty}(\M,\N) \to \R$ defined by
\beq
\label{eq:Jabcobian_functional} 
E(f)=\int_{\M} f^* \dVolh=\int_{\M} \Det df \,\dVolg
\eeq
 is a null-Lagrangian. 
\end{corollary}

\begin{remark}
We limited the formulation of all the statements in this section to compact domains for simplicity. If the domain is non-compact, we need to restrict the discussion to compactly-supported variations, and consider the restriction of the functionals to compact subsets of $\M$. 
\end{remark}

%%%%%%%%%%%%%%%%%%%%%%%%%%%%%%%%%%%%%%%%%%
\subsection{Strong formulation of the Piola identity}
\label{sec:Piola_id_strong}

In this section, we calculate the Euler-Lagrange equation of the Jacobian functional \eqref{eq:Jabcobian_functional}.
For this, we first need to define the coderivative for vector-valued forms.

Let $(\M,\g)$ be a $d$-dimensional oriented Riemannian manifold.
Let $E$ be a vector bundle over $\M$ (of arbitrary rank $n$), endowed with a Riemannian metric $\h$ and a metric affine connection $\nabla^E$. We denote by $\Omega^1(\M;E) = \Gamma(T^*\M\otimes E)$ the space of $1$-forms on $\M$ with values in $E$.
The metrics on $\M$ and $E$ induce a metric on $\Omega^1(\M;E)$, denoted $\innerp{\cdot,\cdot}_{\g,\h}$. 
\begin{definition}
\label{def:coderivative}
The coderivative,
\[
\delta_{\nabla^E}: \Omega^1(\M;E) \to  \Omega^{0}(\M;E) \simeq \Gamma(E)
\]
is the adjoint of the connection $\nabla^E$ with respect to the metric $\innerp{\cdot,\cdot}_{\g,\h}$.
That is, it is defined by the relation
\[
\int_\M \innerp{\sigma,\delta_{\nabla^E} \rho}_{\g,\h}\,\dVolg = \int_\M \innerp{\nabla^E \sigma,\rho}_{\g,\h}\,\dVolg,
\]
for all $\rho\in\Omega^1(\M;E)$ and compactly-supported $\sigma\in \Gamma(E)$.
\end{definition}

\begin{remark}
There exist various explicit formulas  for $\delta_{\nabla^E}$, which we do not mention since they are not used in this work.
\end{remark}

We shall use the coderivative in the following specific setting: 
Let $f:\M\to\N$ be smooth. Its differential is a section $df \in  \Gamma(T^*\M\otimes f^*T\N)=\Omega^1(\M;f^*T\N)$. Set $E=f^*T\N$ and $\nabla^E=\nabla^{f^*T\N}$. Note that $\Cof df  \in \Omega^1(\M;f^*T\N)$ is of the same type as $df$. Hence, $\delta_{\nabla^{f^*T\N}} df$ and $\delta_{\nabla^{f^*T\N}}  \Cof df $ are well-defined according to \defref{def:coderivative}.

\begin{lemma}
\label{lem:E-l_det}
Let $\M$ and $\N$ be $d$-dimensional, smooth, oriented Riemannian manifolds; The Euler-Lagrange equation of
\beq
\label{eq:Jacobian_functional}
E(f)=\int_{\M} f^* \dVolh=\int_{\M} \Det df\, \dVolg
\eeq
 is $\delta_{\nabla^{f^*T\N}} \Cof df = 0$
\end{lemma}

We prove Lemma~\ref{lem:E-l_det} below. Combining \corrref{cor:Null-Lag2} and \lemref{lem:E-l_det} we deduce:

%%%%%%%%%
\begin{proposition}[Piola identity, intrinsic strong formulation]
\label{prop:divCof=0}
Let $(\M,\g)$  and  $(\N,\h)$ be oriented, $d$-dimensional Riemannian manifolds. Let $f\in C^\infty(\M,\N)$.  Then,
\[
\delta_{\nabla^{f^*T\N}} \Cof df = 0.
\]
Equivalently,
 for every compactly supported $\chi\in\Gamma(f^*T\N)$,
\beq
\int_\M \innerp{\Cof df, \nabla^{f^*T\N} \chi}_{\g,\h}\,\dVolg = 0.
\label{eq:Cof_d}
\eeq
\end{proposition}
%%%%%%%%%
Note that this formulation of the Piola identity does not require embedding the target space into a larger Euclidean space. In this sense, it is intrinsic. 
Note also that we do not require here that the manifolds be compact.

Deriving the Euler-Lagrange equation of the Jacobian functional \eqref{eq:Jacobian_functional} essentially amounts to the differentiation of the determinant of a bundle morphism.
It is well-known that the cofactor matrix is the gradient of the determinant. 
We need the following generalized version of this fact in the setting of morphisms between Riemannian vector bundles, whose proof appears in Section~\ref{app:deriv_det_bundle}:
\begin{lemma}
\label{lem:Cofactor_grad_Determinant_bundle}
Let $E$ and $F$ be oriented vector bundles of rank $d$  over a smooth manifold $\M$, equipped with smooth metrics and metric-compatible connections. Let $A:E \to F$ be a smooth bundle map. {Then, for every $V \in \Gamma(\TM)$
 \[
   d(\Det A)(V)= \IP{\Cof A}{\nabla_V A}_{E,F},
 \]
where $\Det A$,$\Cof A$ are as defined in \ref{def:intrinsic_det} and \ref{def:intrinsic_cof}, using the metrics and orientations on $E,F$, and $\nabla A$ is the induced tensor product connection on $E^* \otimes F$ induced by the connections on $E,F$.}
 \end{lemma}

\emph{Proof of \lemref{lem:E-l_det}:}
Let $\phi: \M \to \N$ be a smooth map, and let $V \in \Ga\brk{\phi^*\left( \TN \right)}$. Let $\phi_t:\M \to \N$ be a smooth variation which is constant on $\pl \M$ such that $\phi_0=\phi$ and $\left. \pd{\phi_t}{t}  \right|_{t=0}=V$. Our goal is to prove that
\[
\left. \deriv{}{t} E\left( \phi_t \right) \right|_{t=0}=\int_{\M} \IP{\delta_{\nabla^{\phi^*T\N}} \big(\Cof d\phi \big)}{V}_{\phi^*\TN} \dVolg.
 \]

Denote by $\psi : \M \times I \to \N$ the map $\psi(p,t) = \phi_t\left( p \right)$. Let $P:\M \times I \to \M$ be the projection $P(p,t)=p$. Consider the following vector bundles over $\M \times I$:
\begin{enumerate}
  \item $\brk{P^*\brk{\TM}}^* \cong P^{*}\brk{\TstarM}$. Its fiber over $(p,t)$ is $T^*_p\M$.
  \item $\psi^{*}\brk{\TN}$. Its fiber over $(p,t)$ is $T_{\phi_t(p)}\N$.
\end{enumerate}

Note that $\brk{d\phi_t}_p:T_p\M \to T_{\phi_t(p)}\N$, i.e.~$\brk{d\phi_t}_p \in T^*_p\M \otimes T_{\phi_t(p)}\N$.
Running over all the pairs $(p,t) \in \M \times I$ we obtain a section of the vector bundle $W:= \brk{P^{*}\brk{\TM}}^* \otimes \psi^{*}\brk{\TN}$.

Now,
\beq 
\label{eq:energy_time_deriv_Riemann}
\begin{split}
\left. \deriv{}{t} E\left( \phi_t \right) \right|_{t=0}=&\int_{\M} \left. \deriv{}{t}  \Det (d\phi_t) \right|_{t=0}   \dVolg \teq{(*)} 
\int_{\M} \left.    \IP{\Cof (d\phi_t)}{\nabla_{\pd{}{t}}^W d\phi_t}_{P^{*}\brk{\TM},\psi^{*}\brk{\TN}} \right|_{t=0} \dVolg \\ 
&+
  \int_{\M}   \IP{\Cof d\phi}{\left. \nabla^W_{\pd{}{t}} d\phi_t   \right|_{t=0} }_{\TM,\phi^*\TN} \dVolg ,
\end{split}
\eeq
where equality $(*)$ follows from an application of \lemref{lem:Cofactor_grad_Determinant_bundle} (with $A=d\phi_t,V=\pd{}{t}$).

It is well-known that 
\beq 
\label{eq:symmetric_equality}
\left. \brk{ \nabla^W_{\pd{}{t}} d\phi_t } \right|_{t=0} = \nabla^{\phi^*\TN}V,
\eeq 
See e.g.~\cite[Proposition 2.4, Pg 14]{EL83}.
Eqs. \eqref{eq:energy_time_deriv_Riemann} and \eqref{eq:symmetric_equality} then imply
\[
\left. \deriv{}{t} E\left( \phi_t \right) \right|_{t=0}=\int_{\M}   \IP{\Cof d\phi}{\nabla^{\phi^*\TN}V }_{\TM,\phi^*\TN} \dVolg =\int_{\M} \IP{\delta_{\nabla^{\phi^*T\N}} \big(\Cof d\phi \big)}{V}_{\phi^*\TN} \dVolg,
\]
where the last equality follows from \defref{def:coderivative}.
\hfill\ding{110}

\begin{remark}
In applying \lemref{lem:Cofactor_grad_Determinant_bundle}, we needed the assumption that the connections on $\TM,\TN$ are metric-compatible. More precisely, the Levi-Civita connections on $\TM,\TN$ induce connections on $P^{*}\brk{\TstarM}$ and $\psi^{*}\brk{\TN}$. 
Since the original connections were metric so are the induced ones. 
\end{remark}

%%%%%%%%%%

An immediate corollary of Proposition~\ref{prop:divCof=0} and Corollary~\ref{cor:DetAndCof} is the well-known fact that smooth local isometries between manifolds are harmonic.
Since we want to use the same idea for Lipschitz mappings, we need a weak version of Proposition~\ref{prop:divCof=0} that applies to them. This is Theorem~\ref{thm:weak_Piola_identity}, which is proved in the next section.

%%%%%%%%%%%%%%%%%%%%%%%%%%%%%%%%%%%%%%%%
\subsection{Weak formulation of the Piola identity: Proof of Theorem~\ref{thm:weak_Piola_identity}}
\label{sec:Piola_id_weak}

First, we show that \eqref{eq:Cof_df_orthogonal_extrinsic_a} holds for every $f\in C^\infty(\M,\N)$.  
Given an isometric embedding $\iota:(\N,\h)\to(\R^D,\euc)$,  
\[
d\iota:T\N\to \N\times\R^D 
\Textand f^*d\iota:f^*T\N\to \M\times\R^D. 
\]
Then, Eq.~\eqref{eq:Cof_d} in \propref{prop:divCof=0} can be rewritten as
\beq
\label{eq:divCof=0_aa}
\int_\M\innerp{f^*d\iota\circ\Cof df,f^*d\iota\circ\nabla^{f^*T\N}\chi}_{\g,\euc} \,\dVolg = 0
\eeq
for all $\chi\in\Gamma_0(f^*T\N)$.

Denote by $N\N$ the normal bundle of $\iota(\N)$ in $\R^D$, that is, $N\N\subset\N\times\R^D$ is the orthogonal complement of $d\iota(T\N)$ in $(\N\times\R^D,\euc)$. Denote by $P$ and $P^\perp$ the orthogonal projections of $\N\times\R^D$ into $d\iota(T\N)$ and $N\N$. For a section $\zeta\in\Gamma(T\N)$, the Levi-Civita connection on $T\N$ is induced by the Levi-Civita connection on the trivial bundle, $\N\times\R^D$, by the classical relation
\[
d\iota\circ \nabla^{T\N}\zeta = P\brk{\nabla^{\N\times\R^D} (d\iota\circ\zeta)}
\]
(Recall $\nabla^{\N\times\R^D} \xi$ is simply a componentwise-differentiation of $\xi$).

Let $\zeta\in\Gamma(T\N)$ have compact support in $f(\M)$.
Then, $f^*\zeta\in\Gamma_0(f^*T\N)$, and
\[
\begin{split}
f^*d\iota\circ\nabla^{f^*T\N} f^*\zeta &= f^*\brk{d\iota\circ\nabla^{T\N} \zeta} \\
&= f^* \brk{P\brk{\nabla^{\N\times\R^D} (d\iota\circ\zeta)}} \\
&= (f^* P) \brk{\nabla^{\M\times\R^D} (f^*d\iota\circ f^*\zeta)},
\end{split}
\]
where in the last step we used the fact that $f^*\nabla^{\N\times\R^D} = \nabla^{\M\times\R^D}$.
Since sections of the form $f^*\zeta$ span $\Gamma(f^*T\N)$ locally, it follows that
\beq
\label{eq:divCof=0_ab}
f^*d\iota\circ\nabla^{f^*T\N} \chi = (f^* P)\brk{\nabla^{\M\times\R^D} (f^*d\iota\circ \chi)}.
\eeq

Next, we note that  
\[
f^*d\iota\circ \chi\in\Gamma_0(f^*d\iota(T\N)) \subset \Gamma_0(M\times\R^D).
\]
Sections in $\Gamma_0(f^*d\iota(T\N))$ can be represented by sections in $\Gamma_0(M\times\R^D)$ projected onto $f^*d\iota(T\N)$. That is, setting $f^*d\iota\circ \chi = (f^*P)(\xi)$, and combining \eqref{eq:divCof=0_aa}, \eqref{eq:divCof=0_ab} we get 
\[
\int_\M\innerp{f^*d\iota\circ\Cof df,(f^* P)\brk{\nabla^{\M\times\R^D} (f^*P)(\xi)}}_{\g,\euc} \,\dVolg = 0
\]
for all $\xi\in\Gamma_0(\M\times\R^D)$. 
Since $f^*d\iota\circ\Cof df\in \Gamma(f^*d\iota(T\N))$, the outer projection can be omitted, yielding,
\[
\int_\M\innerp{f^*d\iota\circ\Cof df, \nabla^{\M\times\R^D} (f^*P)(\xi)}_{\g,\euc} \,\dVolg = 0.
\]

Next, set $(f^*P)(\xi) = \xi - (f^*P^\perp)(\xi)$. Then, for all $\xi\in\Gamma_0(\M\times\R^D)$,
\beq
\begin{split}
&\int_\M\innerp{f^*d\iota\circ\Cof df, \nabla^{\M\times\R^D} \xi}_{\g,\euc} \,\dVolg \\
&\qquad= 
\int_\M \tr_\g \innerp{f^*d\iota\circ\Cof df, \nabla^{\M\times\R^D} (f^*P^\perp)(\xi)}_{\euc} \,\dVolg,
\end{split}
\label{eq:second_form_sub_here}
\eeq
where on the right-hand side, we have separated the inner-product on $T^*\M\otimes\R^D$ into, first, an inner-product over $\R^D$, followed by a trace over $T^*\M$. 

Let $A: T\N\times T\N\to N\N$ be the second fundamental form of $\N$ in $\R^D$.
That is, 
\[
\inner{A(u,v),\eta}_\euc = \innerp{d\iota\circ u,\nabla_{v}^{\N\times\R^D} \eta}_\euc,
\]
for $u,v\in\Gamma(T\N)$ and $\eta\in \Gamma(N\N)$. 
Pulling back with $f$,
\[
\innerp{f^*A(u,df(X)),\eta}_\euc = \innerp{f^*d\iota\circ u, \nabla_{X}^{\M\times\R^D} \eta}_\euc,
\]
for $u\in\Gamma(f^*T\N)$, $X\in\Gamma(T\M)$ and $\eta\in \Gamma(f^*N\N)$. Setting $\eta = (f^*P^\perp)(\xi)$ and
$u = \Cof df(X)$,
\[
\innerp{f^*A(\Cof df(X),df(X)),(f^*P^\perp)(\xi)}_\euc = \innerp{f^*d\iota\circ \Cof df(X), \nabla_{X}^{\M\times\R^D} (f^*P^\perp)(\xi)}_\euc.
\]
Since the range of $A$ is $N\N$, the projection $f^*P^\perp$ on the left-hand side can be omitted. Moreover, replacing the vector field $X$ by the components $X_i$ of an orthonormal frame field, and summing over $i$, we obtain
\[
\innerp{\tr_\g f^*A(\Cof df,df),\xi}_\euc = \tr_\g \innerp{f^*d\iota\circ \Cof df, \nabla^{\M\times\R^D} (f^*P^\perp)(\xi)}_\euc.
\]
Substituting this last identity into \eqref{eq:second_form_sub_here}, we finally obtain
\[
\int_\M\innerp{f^*d\iota\circ\Cof df, \nabla^{\M\times\R^D} \xi}_{\g,\euc} \,\dVolg =
\int_\M \innerp{\tr_\g f^*A(\Cof df,df),\xi}_\euc \,\dVolg,
\]
for all $f\in C^\infty(\M,\N)$ and all $\xi\in\Gamma_0(\M\times\R^D)$.

It remains to show that this identity holds for all $f\in W^{1,p}(\M;\N)$ and all $\xi \in W_0^{1,2}(\M;\R^D)\cap L^\infty(\M;\R^D)$. This follows by first approximating $f$ by smooth functions in the $W^{1,p}$ topology (this is possible since $p\ge d$), and then approximating $\xi$ with smooth sections of $\M\times\R^D$ in the $W^{1,2}$ topology.
Since $p\ge 2(d-1)$, then
$f^*d\iota\circ \Cof df\in L^2(\M;T^*\M\otimes \R^D)$, hence the first integrand is well defined for $\xi\in W^{1,2}(\M;\R^D)$. 
Since $p\ge d$, $\tr_\g f^*A(df,\Cof df)\in L^1(\M;\R^D)$, and the second integrand is well-defined for $\xi\in L^\infty(\M;\R^D)$.
The fact that $f_n^*d\iota\circ \Cof df_n\to f^*d\iota\circ \Cof df$ in $L^2$ and $f_n^*A(df_n,\Cof df_n)\to f^*A(df,\Cof df)$ in $L^1$ also hinges on the fact that $p>d$, hence the convergence $f_n\to f$ is uniform.
The necessity of uniform convergence is also the reason for assuming $p>2$ for $d=2$, rather than $p\ge 2(d-1)=2$.
\hfill\ding{110}

%%%%%%%%%%%%%%%%%%%%%%%%%%%%%%%%%%%%%%%%%%%%%%%%%%%%%%%%%%
\subsection{Coordinate formulation of the Piola identity}
\label{sec:Piola_id_coor}
For completeness, we formulate the strong and weak Piola identities in local coordinates:
Let the indices $i,j,k$ denote coordinates on $\M$, $\alpha,\beta,\gamma$ denote coordinates on $\N$ and $a,b$ denote coordinates on $\R^D$. 
$\g_{ij}$ and $\h_{\alpha\beta}$ denote the entries of the metrics $\g$ and $\h$, respectively, and $\Gamma^\alpha_{\beta\gamma}$ are the Christoffel symbols of $\nabla^{\N}$.
The differential $df$, consists of vectors $\pl_i f\in T\N$ that have entries $\pl_i f^\alpha$; similarly, $\Cof df=(\Cof df)_i^\alpha$.
Then the strong Piola identity \eqref{eq:Cof_d} reads
\beq
\label{eq:Cof_d__coordinates}
\int_\M (\Cof df)_i^\alpha \,\g^{ij} \h_{\alpha\beta}\, \brk{\pl_j \xi^\beta +\pl_j f^\gamma \Gamma^\beta_{\gamma\delta} \xi^\delta}\sqrt{|\g|}\,dx = 0,
\eeq
where both $\h$ and $\Gamma$ are evaluated at $f(x)$.

The weak Piola identity \eqref{eq:Cof_df_orthogonal_extrinsic_a} reads 
\beq
\label{eq:Piola_id_weak_coor}
\int_\M  \g^{ij}\, \partial_\alpha \iota^a (\Cof df)_i^\alpha\, \delta_{ab}\, \pl_j \xi^b \,\sqrt{|\g|}\,dx
=
\int_\M \g^{ij}\, A^a((\Cof df)_i, \pl_j f)\, \delta_{ab}\, \xi^b \,\sqrt{|\g|}\,dx,
\eeq
where $|\g| = \det \g_{ij}$, $A^a$ are the entries of the second fundamental form induced by $\iota$, and both
$\iota$ and $A$ are evaluated at $f(x)$.

The cofactor operator $\Cof df$ reads in coordinates
\beq
\label{eq:CofAformula_coordinates}
(\Cof df)_i^\alpha = \frac{\sqrt{|\h|}}{\sqrt{|\g|}} h^{\alpha\beta} \delta_{\beta\gamma} (\cof df)_k^\gamma \delta^{kj} \g_{ji},
\eeq
where $\cof df$ is the cofactor of the matrix $\pl_i f^\alpha$ (see \eqref{eq:CofAformula} in Proposition \ref{lm:Cof_vs_cof}).

%%%%%%%%%%%%%%%%%%%%%%%%%%%%%%%%%%%%%%%%%%%%%%%%%%%%%%%%%%%%%%%%%%%%%%%%%%%%%%%%%%%%%%%%%%%%%%%%%%%%%%%%%%%%%%%%%%%%%%%%%%%%%%

\section{Reshetnyak's rigidity theorem for manifolds}
\label{sec:asymptotic_rigidity}

In this section we prove Theorem~\ref{thm:Reshetnyak_M_N}.
Before the proof we state a version of the fundamental theorem of Young measures, adapted to our setting, which will be used throughout the proof.

\subsection{Young measures on vector bundles}
\label{sec:Young_meas}

The following theorem is an adaptation of the fundamental theorem of Young measures \cite{Bal89}, adapted from Euclidean settings to vector bundles with Riemannian metrics (more generally, it applies to any Finsler vector bundle).

%%%%%%%%%%%
\begin{theorem}
\label{thm:Young_meas_bundles}
Let $(\M,\g)$ be a compact Riemannian manifold. Let $E\to \M$ be a vector bundle endowed with a Riemannian metric.
Let $(\xi_n)$ be a sequence of measurable sections of $E$, bounded in $L^1(\M;E)$.
Then, there exists a subsequence $(\xi_n)$ and a family $(\nu_x)_{x\in\M}$ of Radon probability measures on $E_x$, depending measurably on $x$, such that
\beq
\label{eq:young_measure_theorem_L1convergence}
\psi\circ\xi_n \weakly \BRK{x\mapsto \int_{E_x} \psi_x(\lambda)\, d\nu_x(\lambda)}
\quad \text{in $L^1(\M;W)$},
\eeq
for every Riemannian vector bundle $W\to\M$ and every continuous bundle map $\psi:E\to W$ (not necessarily linear), satisfying that $(\psi\circ\xi_n)$ is sequentially weakly relatively compact in $L^1(\M;W)$.
\end{theorem}
%%%%%%%%%%%

\begin{remark}
The criterion that $(\psi\circ \xi_n)$ is sequentially weakly relatively compact in $L^1(\M;W)$ is equivalent to
\beq
\label{eq:cond_L1_weak_comp}
\sup_n \int_\M \vp(|\psi\circ \xi_n|) \, \dVolg < \infty
\eeq
for some continuous function $\vp: [0,\infty)\to \R$, such that $\lim_{t\to\infty} \vp(t)/t =\infty$.
This is known as the de la Vall\'ee Poussin's criterion \cite[Remark 3]{Bal89}.
\end{remark}

The above theorem makes use of the following definitions. 
\begin{definition}
Let $(\M,\g)$ be a compact Riemannian manifold. Let $W,E\to \M$ be Riemannian vector bundles.
\begin{enumerate}
  \item The space $C_0(E;W)$ is the space of continuous bundle maps (not necessarily linear) $E\to W$ that are decaying fiberwise. That is, if $h\in C_0(E;W)$, then for every $x\in\M$,  
\[
\lim_{E_x\ni e\to \infty} |h_x(e)|_{W_x} = 0.
\]

\item $M(E)$ is the bundle of bounded Radon  measures on $E$. A section $\mu$ of $M(E)$ is \emph{measurable} (more accurately,  \emph{weak-$*$-measurable}) if for every bundle map $f\in C_0(E;\R)$, the real-valued function
\[
\BRK{x\mapsto \int_{E_x} f_x(e)\, d\mu_x(e)} : \M\to\R
\]
is measurable;
note that this implies the measurability of 
\[
\BRK{x\mapsto \int_{E_x} f_x(e)\, d\mu_x(e)} : \M\to W
\]
for every $f\in C_0(E;W)$.

\item 
$h_n\to h$ weakly in $L^1(\M;W)$ if for every $\phi\in L^\infty(W^*)$,
\[
\int_\M \phi\circ h_n \,\,\dVolg \to \int_\M \phi\circ h \,\,\dVolg
\]
where $L^\infty(W^*)$ is the space of essentially bounded measurable vector bundle morphisms $\phi:W \to \M\times\R$.
Note that while, generally, the composition of measurable functions is not measurable, in this case the fiberwise linearity of $\phi$ implies that the composition amounts to a scalar multiplication of vectors, which is measurable.
\end{enumerate}
\end{definition}

The proof of Theorem~\ref{thm:Young_meas_bundles} follows the lines of the proof of the Euclidean case \cite{Bal89}, with some natural adaptations.
Mainly, by taking a partition of $\M$ fine enough such that there exist local trivializations of $W$ subordinate to the partition, one can follow the proof of \cite{Bal89} on each element of the trivialization.
Note that we cannot use a similar approach to reduce the proof of Theorem~\ref{thm:Reshetnyak_M_N} to the proof of the  Euclidean Reshetnyak theorem. 
One reason is that Sobolev spaces between manifolds (which are the spaces considered in Theorem~\ref{thm:Reshetnyak_M_N}), are more complicated than Sobolev spaces on vector bundles (considered in Theorem~\ref{thm:Young_meas_bundles}); 
in particular, they are not vector spaces, and smooth functions are not necessarily dense subspaces (see Appendix~\ref{subsec:sobolev} for details).
Therefore, an adaptation of the proof of \cite{JK90} is more delicate.

%%%%%%%%%%%%%%%%%%%%%%%%%%%%%%%%%%%%%%%%%%%%%%%%%%%%%%%%%
\subsection{Proof of Theorem~\ref{thm:Reshetnyak_M_N}}
\label{sec:Reshetnyak_proof}
Since the proof is long, we divide it into six steps.
In Steps~I--III, we assume that $p>d$;
in Step~I, we show that $f_n$ converges uniformly to $f$;
in Step~II, we show that $f$ is an isometric immersion (this main step is the one most reminiscent to \cite{JK90}, after an appropriate localization of the problem);
in Step~III, we show that  $f_n$ converges to $f$ also in the strong topology of $W^{1,p}(\M;\N)$.
In Step~IV we relax the $p>d$ assumption, and prove that the results of Steps~I-III hold for $p\ge1$ (note that \cite{JK90} does not treat this case even in the simpler Euclidean settings). 
Finally,  in Steps~V-VI we prove that $f$ is an isometry if the additional assumption on $f_n$ and the equality of volumes are satisfied.

Remark about notation: 
In Theorem~\ref{thm:Reshetnyak_M_N} the distance $\dist(df_n,\SO{\g,f^*_n\h})$ is calculated with respect to the inner-product induced on $T^*\M\otimes f_n^*T\N$ by $\g$ and $\h$.
Throughout the proof there occur similar expressions, each using a different inner-product to calculate distances.
To keep track of which inner-product is being used, we will refer to it in the subscript, e.g., $\dist_{(\g,f_n^*\h)}(df_n,\SO{\g,f^*_n\h})$.

%%%%%
\bigskip
\emph{Step I: $f_n$ has a uniformly converging subsequence}

As described in Appendix~\ref{subsec:sobolev}, Sobolev maps between manifolds are defined by first embedding the target manifold isometrically into a higher-dimensional Euclidean space. 
Let $\iota:(\N,\h)\to (\R^D,\euc)$ be a smooth isometric embedding of $\N$, where $\euc$ denotes the standard Euclidean metric on $\R^D$, and let 
\[
F_n = \iota \circ f_n:\M\to\R^D
\]
be the ``extrinsic representative" of $f_n$.
For $x\in\M$,
denote by $O(\g_x,\euc)$ the set of linear isometric embeddings $(T_x\M,\g_x)\to (\R^D,\euc)$. Note that when mapping a vector space into a vector space of higher dimension, there is no notion of preservation of orientation; in particular, $\SO{\g,\euc}$ is not defined. However, since $A\in \O{\g_x,\h_{f(x)}}$ implies that 
$d\iota_{f(x)}\circ A \in \O{\g_x,\euc}$, it follows that
\[
\begin{split}
  \dist_{(\g,\euc)}(d F_n, \O{\g,\euc}) &\le \dist_{(\g,f_n^*\h)}(d f_n, \O{\g,f_n^*\h}) \\
&\le \dist_{(\g,f_n^*\h)}(d f_n, \SO{\g,f_n^*\h}).
\end{split}
\]
In particular, \eqref{eq:Lpconv_of_dist} implies that
\beq
\label{eq:dist_dF_n_O}
\dist(d F_n, \O{\g,\euc})\to 0 \qquad\text{in $L^p(\M)$}.
\eeq
Since, by the compactness of $\N$,  $\image(F_n) \subseteq \iota(\N) \subset\R^D$ is bounded, it follows from the Poincar\'e inequality that $F_n$ are uniformly bounded in $W^{1,p}(\M;\R^D)$. Hence, $F_n$ has a subsequence  converging weakly in $W^{1,p}(\M;\R^D)$ to a limit $F\in W^{1,p}(\M;\R^D)$.

Since $p>d$, it follows from the Rellich-Kondrachov theorem \cite[Theorem~6.3]{AF03} that $F_n \to F$ uniformly; since $\iota(\N)$ is closed in $\R^D$, it follows that $F(\M)\subset \iota(\N)$, i.e., $\iota^{-1}\circ F:\M\to\N$ is well-defined. The compactness of $\N$ implies that the intrinsic and the extrinsic distances on $\N$ are strongly equivalent (see \cite{MC1}). Therefore, $f_n \to \iota^{-1}\circ  F$ uniformly; we denote this limit by $f$; it is in $W^{1,p}(\M;\N)$ by the very definition of that space.

%%%%%
\bigskip
\emph{Step II: $f$ is an isometric immersion}

By Theorem~\ref{thm:Liouville_M_N_Lipschitz}, it is sufficient to prove that $df \in \SO{\g,f^*\h}$ a.e.
Note that this is a local statement. Thus, it suffices to show that every $x\in\M$ has an open neighborhood in which this property holds. Using local coordinate charts, this statement can be reformulated in terms of mappings between a manifold and a Euclidean space of the same dimension; as already discussed, the equality of dimension is critical for keeping track of orientation-preserving linear maps. 

So let $x\in \M$ and let $\phi:U \subset\R^d \to \N$ be a positively-oriented coordinate chart around $f(x)\in \N$.
Let $\M'$ be an open neighborhood of $x$ such that $\overline{f(\M')}\subset \phi(U)$. 
Since $f_n\to f$ uniformly, and the distance between $f(\M')$ and the boundary of $\phi(U)$ is positive, $f_n(\M')\subset \phi(U)$ for $n$ large enough. 

In the rest of this step of the proof, we will view $f_n$ and $f$ as mappings $\M'\to U\subset\R^d$; for $y\in U$,  $T_yU\simeq \R^d$ will be endowed with either the Euclidean metric $\euc$ or the pullback metric $\phi^\star\h$, with entries $\h_{ij}(y) = \h(\partial_i,\partial_j)|_{\phi(y)}$.
Since we can assume that $f_n(\M')$ are all contained in the same compact subset of $\phi(U)$, it follows that we can assume that all the entries $\h_{ij}$ and $\h^{ij}$ of the metric and its dual are uniformly continuous, and in particular uniformly bounded. 

The uniform boundedness of $\h_{ij}$ and $\h^{ij}$ implies that the norms on $T\M'\otimes\R^d$ and $T^*\M'\times \R^d$ induced by (i) $\g$ and $f_n^*\h$, (ii) $\g$ and $f^*\h$, and (iii) $\g$ and $\euc$ are all equivalent; moreover, the constants in these equivalences are independent of both $n$ and $x$. 

This implies that both weak and strong convergence in $L^q(\M';T^*\M\otimes \R^d)$ are the same with respect to either of those norms.

As distances in $T\M'\otimes\R^d$ with respect to $(\g,f_n^*\h)$ and $(\g,f^*\h)$ are equivalent, \eqref{eq:Lpconv_of_dist} implies that
\[
\dist_{(\g,f^*\h)}(d f_n, \SO{\g,f_n^*\h})\to 0 \qquad\text{in $L^p(\M')$}.
\]

The uniform boundedness of entries of $\h$ along with the uniform convergence of $f_n$ to $f$ implies that 
\[
\dist_{(\g,f^*\h)}(\SO{\g,f_n^*\h}, \SO{\g,f^*\h}) \to 0
\]
uniformly in $\M'$, where the distance here is the Hausdorff distance induced by $\dist_{(\g,\euc)}$. 
Hence 
\beq
\dist_{(\g,f^*\h)}(d f_n, \SO{\g,f^*\h})\to 0 \qquad\text{in $L^p(\M')$}.
\label{eq:df_n_to_SO_g_fh}
\eeq
Comparing  \eqref{eq:df_n_to_SO_g_fh} and \eqref{eq:Lpconv_of_dist}, we replaced the $n$-dependent set $\SO{\g,f_n^*\h}$ by the fixed set $\SO{\g,f^*\h}$ and the $n$-dependent metric induced by $\g$ and $f^*_n\h$ by the fixed metric induced by $\g$ and $f^*\h$. 
It follows from  \eqref{eq:df_n_to_SO_g_fh} that $df_n$ is uniformly bounded 
in $L^p(\M';T^*\M'\otimes\R^d)$.
Since, moreover, $f_n(\M')$ is uniformly bounded in $\R^d$, 
it follows that $f_n$ has a subsequence that weakly converges in $W^{1,p}(\M';\R^d)$.
Since weak convergence in $W^{1,p}(\M';\R^d)$ implies uniform convergence, the limit coincides with $f$.

Henceforth, denote by $E$ the vector bundle $T^*\M'\otimes \R^d$ with the metric induced by $\g$ and $f^*\h$.
Note that we view all the mappings $df_n$ as sections of the same vector bundle $E$, which is the key reason for using a local coordinate chart for $\N$.

The sequence $df_n$ satisfies the conditions of Theorem~\ref{thm:Young_meas_bundles}, including the boundedness in $L^1$ (since $df_n$ are bounded in $L^p$ and $\Vol(\M) < \infty$). Hence, there exists a subsequence $f_n$, and a family of Radon probability measures $(\nu_x)_{x\in\M'}$ on $E_{x}$, such that
\beq
\label{eq:young_meas_result_bundles}
\psi \circ df_n \weakly \brk{x\mapsto 
\int_{E_x} \psi_x(\lambda)\, d\nu_x(\lambda)} \quad \text{in}\,\,L^1(\M';W)
\eeq
for every Riemannian vector bundle $W\to\M'$ and every continuous bundle map $\psi:E\to W$, such that $\psi\circ df_n$ is sequentially weakly relatively compact in $L^1(\M';W)$. 
The idea  is to exploit the general relation \eqref{eq:young_meas_result_bundles} for various choices of $W$ and $\psi$.

First, consider \eqref{eq:young_meas_result_bundles} for $W=\R$ and $\psi = \dist_{(\g,f^*\h)}(\cdot, \SO{\g,f^*\h})$. The compactness condition is satisfied since $\psi\circ d f_n$ is bounded in $L^{p}(\M';\R)$ and $p>1$ \cite[p.~68]{LL01}.
We obtain that
\beq
\label{eq:young_meas_in_SO}
\dist_{(\g,f^*\h)}(df_n, \SO{\g,f^*\h}) \weakly \brk{x\mapsto \int_{E_x} \dist_{(\g,f^*\h)}(\lambda, \SO{\g,f^*\h})|_x\,d\nu_x(\lambda)}
\eeq
in $L^1(\M';\R)$. Multiplying by the test function $1 \in L^{\infty}(\M;\R)$ and integrating over $\M'$ we obtain
\[
\begin{split}
0 &= \lim_n \int_{\M'} \dist_{(\g,f^*\h)}(df_n, \SO{\g,f^*\h})\,\dVolg \\
&= \int_{\M'} \brk{\int_{E_x} \dist_{(\g,f^*\h)}(\lambda, \SO{\g,f^*\h})|_x\,d\nu_x(\lambda)} \,\dVolg|_x.
\end{split}
\]
This implies that $\nu_x$ is supported on $\SO{\g,f^*\h}_x$ for almost every $x\in \M'$.

Next, consider \eqref{eq:young_meas_result_bundles} for the following choices of $W$ and $\psi$,
\[  
\begin{array}{ll}
W=E &\psi=Id \\
W = \R &\psi=\Det \\
W =E &\psi = \Cof,
\end{array}
\]
where the determinant and the cofactor are defined with respect to the metric induced by $\g$ and $f^*\h$ (see Section~\ref{sec:det_cof} for intrinsic definitions of the determinant and the cofactor).
Since $p> d$, all three choices of $\psi$ imply that $\psi\circ df_n$ satisfy the $L^1$-weakly sequential compactness condition.

Therefore,
\beq
\label{eq:young_meas_det_cof_manifolds}
\begin{array}{rll}
df_n &\weakly  \brk{x\mapsto \int_{E_x} \lambda \,d\nu_x(\lambda)}  & \qquad \text{in}\,\,L^1(\M;E)   \\
\Det(df_n) &\weakly \brk{x\mapsto\int_{E_x} \Det(\lambda)|_x \,d\nu_x(\lambda)}  & \qquad\text{in}\,\, L^1(\M';\R)    \\
\Cof(df_n) &\weakly \brk{x\mapsto\int_{E_x} \Cof(\lambda)|_x \,d\nu_x(\lambda)} & \qquad\text{in}\,\, L^1(\M';E),
\end{array}
\eeq
where the dependence of $\Det(\lambda)$ and $\Cof(\lambda)$ on $x$ is via the metrics $\g$ and $f^*\h$.
Since $\nu_x$ is supported on $\SO{\g,f^*\h}$, and $\Det(\lambda)=1$ and $\Cof(\lambda)=\lambda$ for $\lambda\in \SO{\g,f^*\h}$, \eqref{eq:young_meas_det_cof_manifolds} reduces to
\beq
\label{eq:young_meas_det_cof_manifolds2}
\begin{array}{rll}
df_n &\weakly  \brk{x\mapsto \int_{\SO{\g,f^*\h}_x} \lambda \,d\nu_x(\lambda)}  &\qquad \text{in}\,\,L^1(\M';E)   \\
\Det(df_n) &\weakly 1  &\qquad \text{in}\,\, L^1(\M;\R)    \\
\Cof(df_n) &\weakly \brk{x\mapsto \int_{\SO{\g,f^*\h}_x} \lambda \,d\nu_x(\lambda)} &\qquad \text{in}\,\, L^1(\M';E).
\end{array}
\eeq

On the other hand, by the weak continuity of determinants and cofactors (see Proposition \ref{pn:abstract_weak_continuity_det_cof}),
\beq
\label{eq:lim_df_n_det_cof}
\begin{array}{rll}
	df_n &\weakly df  & \qquad\text{in}\,\, L^p(\M';E) \\
	\Det(df_n) &\weakly \Det(df) &  \qquad\text{in}\,\, L^{p/d}(\M';\R) \\
	\Cof(df_n) &\weakly \Cof(df) &\qquad \text{in}\,\, L^{p/(d-1)}(\M';E)
\end{array}
\eeq
Combining \eqref{eq:young_meas_det_cof_manifolds2} and \eqref{eq:lim_df_n_det_cof}, it follows from the uniqueness of the limit in $L^1$ that the following hold almost everywhere
\beq
\label{eq:df_in_SO_1}
\begin{split}
	df(x) &= \int_{\SO{\g,f^*\h}_x} \lambda \,d\nu_x(\lambda) \\
	\Det(df(x)) &= 1 \\
	\Cof(df(x)) &= \int_{\SO{\g,f^*\h}_x} \lambda \,d\nu_x(\lambda) = df(x).
\end{split}
\eeq
Since $\Cof(df) = df$ and $\Det(df)=1$ a.e., it follows from Corollary~\ref{cor:DetAndCof} that 
\beq
df\in \SO{\g,f^*\h} \qquad \text{a.e.}
\label{eq:df_in_SO_2}
\eeq
By \thmref{thm:Liouville_M_N_Lipschitz}, it follows that $f:\M \to \N$ is smooth as a map between manifolds with boundary.

%%%%%
\bigskip
\emph{Step III: $f_n\to f$ in the strong $W^{1,p}(\M;\N)$ topology}

We have thus far obtained that $f_n\to f$ uniformly,  that $f$ is a.e. differentiable and $df\in \SO{\g,f^*\h}$. 
We proceed to show that $f_n\to f$ (strongly) in $W^{1,p}(\M;\N)$.

As in Step~I, let $\iota:\N\to\R^D$ be a smooth isometric embedding and let $F_n = \iota\circ f_n$ and $F = \iota\circ f$.
By definition, $f_n\to f$ in $W^{1,p}(\M;\N)$ if $F_n\to F$ in $W^{1,p}(\M;\R^d)$ (Appendix~\ref{subsec:sobolev}).

We repeat a similar analysis as in Step~II for the sections $dF_n$ of $T^*\M\otimes \R^D$.
We obtain a family of Young probability measures $(\mu_x)_{x\in\M}$ on $T_x^*\M\otimes \R^D$ that correspond to $dF_n$.
That is,
\beq
\label{eq:young_meas_result_bundles_extrinsic}
\psi \circ dF_n \weakly \brk{x\mapsto 
\int_{T_x^*\M\otimes \R^D} \psi_x(\lambda)\, d\mu_x(\lambda)} \quad \text{in}\,\,L^1(\M;W)
\eeq
for every Riemannian vector bundle $W\to\M$ and every continuous bundle map $\psi:T^*\M\otimes \R^D\to W$, such that $\psi\circ dF_n$ is sequentially weakly relatively compact in $L^1(\M;W)$. 

Since $\|\dist(dF_n,\O{\g,\euc})\|_p\to 0$ (see \eqref{eq:dist_dF_n_O}), we obtain, by an analysis similar to that leading to \eqref{eq:young_meas_in_SO}, that $\mu_x$ is supported on $\O{\g_x,\euc}$ for almost every $x\in \M$.
As in \eqref{eq:df_in_SO_1}, we obtain
\[
dF_x = \int_{\O{\g,\euc}_x} \lambda \,d\mu_x(\lambda), \qquad\text{a.e.}
\]
We also know that $f$ is an isometric immersion, hence $dF\in \O{\g,\euc}$.
Since $\mu_x$ is a probability measure, we have just obtained that an element in  $\O{\g,\euc}_x$ is equal to a convex combination of elements in $\O{\g,\euc}_x$.
However, $\O{\g,\euc}_x$ is a subset of the sphere of radius $\sqrt{d}$ around the origin in $T_x^*\M\otimes \R^D$, and therefore, it is strictly convex.
It follows that the convex combination must be trivial, namely,
\[
\mu_x = \delta_{dF_x} \quad \text{a.e. in}\,\,\M,
\]
which together with \eqref{eq:young_meas_result_bundles_extrinsic} implies that
\[
\psi \circ dF_n \weakly \psi \circ dF \quad \text{in}\,\,L^1(\M;W).
\]

If we could take for $W = \R$ and $\xi\in T^*\M\otimes \R^D$,
\[
\psi(\xi) = |\xi - dF|^p,
\]
then we would be done, however, this function does not satisfy the sequential weak relative compactness condition. 
Hence, let
\[
\psi(\xi) = |\xi - dF|^p\, \vp\brk{\frac{|\xi|}{3\sqrt{d}}},
\] 
where $\vp:[0,\infty)\to[0,1]$ is continuous, compactly-supported and satisfies $\vp(t)=1$ for $t\le 1$ and $\vp(t)<1$ for $t>1$. This choice of $\psi$ satisfies the de la Vall\'ee Poussin criterion \eqref{eq:cond_L1_weak_comp}, and therefore  
\[
\psi\circ d F_n \weakly 0
\]
in $L^1(\M)$.
In particular, taking the test function $1\in L^\infty(\M)$,
\beq
\label{eq:psi_r_to_zero_manifolds_extrinsic}
\lim_n \int_{\M} \psi\circ dF_n\,\dVolg = 0.
\eeq

We now split the integral in \eqref{eq:psi_r_to_zero_manifolds_extrinsic} into integrals over two disjoint sets, $\M_n$ and $\M_n^c$, where
\[
\M_n = \{x\in\M ~:~ |(dF_n)_x| \le 3\sqrt{d}\}.
\]
By the definition of $\vp$, 
\[
\psi\circ dF_n = |dF_n - dF|^p
\qquad \text{in $\M_n$}.
\]
On the other hand, in $\M_n^c$,
\beq
\begin{split}
\label{eq:dist_large_matrix_SO_inequality_extrinsic}
|dF_n - dF| &\le |dF_n| + |dF| = |dF_n| + \sqrt{d} \le 2(|dF_n| - \sqrt{d}) \le 2\,\dist_{\g,\euc}(dF_n,\O{\g,\euc}),
\end{split}
\eeq
where the last inequality follows from the reverse triangle inequality. 

Combining \eqref{eq:psi_r_to_zero_manifolds_extrinsic} and \eqref{eq:dist_large_matrix_SO_inequality_extrinsic}, 
\[
\begin{split}
&\limsup_{n\to \infty} \int_{\M} |dF_n - dF|^p\, \dVolg \\
&\qquad = \limsup_{n\to \infty} \brk{\int_{\M_n} |dF_n - dF|^p\, \dVolg +  \int_{\M_n^c} |dF_n - dF|^p\, \dVolg} \\
&\qquad \le \limsup_{n\to \infty} \int_{\M} \psi \circ dF_n\, \dVolg + \limsup_{n\to \infty} \int_{\M_n^c} |dF_n - dF|^p\, \dVolg\\
&\qquad = \limsup_{n\to \infty} \int_{\M_n^c} |dF_n - dF|^p\, \dVolg\\
&\qquad \le \limsup_{n\to \infty} 2^p\int_{\M_n^c} \dist_{\g,\euc}^p(dF_n,\O{\g,\euc})\, \dVolg\\
&\qquad \le \limsup_{n\to \infty} 2^p\int_{\M} \dist_{\g,\euc}^p(dF_n,\O{\g,\euc})\, \dVolg = 0,
\end{split}
\]
where the last equality follows from \eqref{eq:dist_dF_n_O}.
Therefore, $dF_n \to dF$ in $L^p(\M;T^*\M\otimes \R^D)$. Since $F_n$ converges uniformly to $F$, we get that $F_n\to F$ in $W^{1,p}(\M;\R^D)$, and, by definition, $f_n\to f$ in $W^{1,p}(\M;\N)$.

%%%%%
\bigskip
\emph{Step IV: Extension to  $1\le p\le d$}

Suppose now that $p\ge 1$. The idea is to replace the functions $f_n$ by functions $f'_n$ that are more regular (specifically, uniformly Lipschitz), and then apply Steps I--III to the approximate mappings $f'_n$.

As in Step I of the proof, we choose a smooth isometric embedding $\iota:(\N,\h)\to (\R^D,\euc)$, and set $F_n = \iota \circ f_n:\M\to\R^D$.
Our assumptions on $f_n$ imply that $F_n\in W^{1,p}(\M;\R^D)$ (this is how $W^{1,p}(\M;\N)$ is defined), and  
\[
\dist(d F_n, \O{\g,\euc})\to 0 \qquad\text{in $L^p(\M)$}.
\]
As in Step~I, it follows that $dF_n$ has a weakly converging subsequence, and together with the Poinrcar\'e inequality, implies that $F_n$ has a subsequence weakly converging in $W^{1,p}(\M;\R^D)$. However, since $p<d$, convergence is not uniform, and the limit does not necessarily lie in the image of $\iota$.

To overcome this problem, we approximate the mappings $F_n$ by another sequence $F_n'\in W^{1,\infty}(\M;\N)$, using the following truncation argument \cite[Proposition A.1]{FJM02b}:

\begin{quote}
\emph{
Let $p\ge1$. 
There exists a constant $C$, depending only on $p$ and $\g$, such that for every $u\in W^{1,p}(\M;\R^D)$ and every $\lambda>0$, there exists $\tilde{u}\in W^{1,\infty}(\M;\R^D)$ such that  
\[
\|d\tilde{u}\|_\infty\le C\lambda, 
\]
\[
\Volg\brk{\BRK{x\in\M: \tilde{u}(x)\ne u(x) }} \le \frac{C}{\lambda^p}\int_{\BRK{|du(x)|>\lambda}} |du|^p \,\dVolg,
\]
\[
\|d\tilde{u}-du\|^p_p\le C\int_{\BRK{|du(x)|>\lambda}} |du|^p \,\dVolg.
\]
}
\end{quote}

The original proposition (\cite[Proposition A.1]{FJM02b}) refers to a bounded Lipschitz domain in $\R^d$, but the partition of unity argument used to obtain the result for an arbitrary Lipschitz domain (Step 3 in the proof) applies to any compact Riemannian manifold with Lipschitz boundary (the constant $C$ depends on the manifold, of course).

Let $\lambda>2\sqrt{d}$, so that $|A|>\lambda$, for $A\in \TstarM_x\otimes \R^D$, implies that $|A|<2\dist(A, \O{\g_x,\euc})$ (compare with \eqref{eq:dist_large_matrix_SO_inequality_extrinsic}). Applying the truncation argument to $F_n$,
we obtain mappings $\tilde{F}_n\in W^{1,\infty}(\M;\R^D)$, with a uniform Lipschitz constant $C$, such that
\beq
\begin{split}
  \Volg\brk{\BRK{x\in\M: \tilde{F}_n(x)\ne F_n(x) }} &\le C\int_{\BRK{|dF_n}  (x)|>\lambda} |dF_n|^p \,\dVolg \\
  &\le C\int_{\BRK{| dF_n  (x)|>\lambda}}  2^p\dist^p(dF_n, \O{\g,\euc}) \,\dVolg \\
	&\le  2^p  C\int_{\M} \dist^p(dF_n, \O{\g,\euc}) \,\dVolg,
\end{split}
\label{eq:F_n_ne_tilde_F_n}
\eeq
and
\beq
\label{eq:dF_n_minus_d_tilde_F_n}
\|d\tilde{F}_n-dF_n\|^p_p\le  2^p C\int_{\M} \dist^p(dF_n, \O{\g,\euc}) \,\dVolg.
\eeq
for some $C>0$, independent of $n$. In particular,
\beq
\lim_{n\to\infty} \Volg\brk{\{x\in\M: \tilde{F}_n(x)\ne F_n(x) \}} = 0
\Textand
\lim_{n\to\infty} \|d\tilde{F}_n-dF_n\|^p_p = 0.
\label{eq:replaces_both}
\eeq

Since $\dist(d F_n, \O{\g,\euc})\to 0$ in $L^p$, \eqref{eq:F_n_ne_tilde_F_n} implies that for every $\e>0$ and every ball $B\subset \M$ of radius $\e$, there exists, for sufficiently large $n$, a point $x\in B$ such that $\tilde F_n(x)=F_n(x)\in \iota(\N)$. 
Since $\M$ is compact and $\tilde{F}_n$ are uniformly Lipschitz, it follows that for large enough $n$, $\max_{x\in \M} \dist(\tilde{F}_n(x), \iota(\N)) < C\e$, and therefore
\[
\max_{x\in \M} \dist(\tilde{F}_n(x), \iota(\N))\to 0.
\]

Thus, for $n$ large enough, $\tilde{F}_n$ lies in a tubular neighborhood of $\iota(\N)$, in which the orthogonal projection $P$ onto $\iota(\N)$ is well-defined and smooth (and in particular Lipschitz).
We define $F'_n:= P\circ \tilde{F}_n$.
It immediately follows that $F'_n\in W^{1,\infty}(\M;R^D)$ are uniformly Lipschitz, and by definition, their image is in $\iota(\N)$.
Moreover, since 
\[
\{F'_n \ne \tilde{F}_n\}= \{\tilde{F}_n\notin \iota(\N)\} \subset \{\tilde{F}_n\ne F_n \} \textand
\{F'_n \ne F_n\} \subset \{\tilde{F}_n\ne F_n \},
\]
\eqref{eq:replaces_both} implies that
\beq
\begin{split}
\lim_{n\to\infty} \Volg\brk{\{F'_n\ne \tilde{F}_n \}} = \lim_{n\to\infty}  \Volg\brk{\{F'_n\ne F_n \}} = 0.
\end{split}
\label{eq:F_n_ne_F_prime_n}
\eeq
Since $\tilde F_n$ and $F'_n$ are uniformly Lipschitz, it follows that
\[
  \| dF'_n - d \tilde F_n\|_p^p = \int_{\BRK{F'_n\ne \tilde{F}_n }} |dF'_n - d \tilde F_n|^p\, \dVolg \le C \Volg\brk{\{F'_n\ne \tilde{F}_n \}}\to 0,
\]
where we used the fact that $dF'_n-d\tilde{F}_n=0$ almost everywhere on the set $F'_n-\tilde{F}_n=0$ (see \cite[Theorem 4.4]{CG15}).

Together with  \eqref{eq:replaces_both} we obtain that $\| dF'_n - d  F_n\|_p\to 0$.
Finally, since
\[
  \int_\M |F'_n(x)-F_n(x)|^p\,\dVolg \le 2^p\max_{y\in \N}|\iota(y)|^p\, \Volg\brk{\{F'_n\ne F_n \}}\to 0,
\]
we conclude that
\beq
\label{eq:F_prime_n_minus_F_n_in_Sobolev}
\lim_{n\to\infty}  \|F'_n - F_n\|_{W^{1,p}(\M;\R^D)} = 0.
\eeq

Next, define $f'_n=\iota^{-1}\circ F'_n$. 
By definition $f'_n\in W^{1,\infty}(\M;\N)$, and moreover, $f'_n$ are uniformly Lipschitz (since intrinsic and extrinsic distances in $\iota(\N)$ are equivalent).
Since $dF'_n=dF_n$ almost everywhere on the set $\BRK{F'_n= {F}_n }$ (again, \cite[Theorem 4.4]{CG15}), we have that $df'_n=df_n$ almost everywhere in the set $\BRK{f'_n= f_n }$.
Using the uniform bound on $df'_n$, we obtain
\beq
\begin{split}
&\int_\M \dist^p_{(\g,{f'}_n^*\h)}(d f'_n, \SO{\g,{f'}_n^*\h}) \, \dVolg \\
&\quad \le \int_{\BRK{f'_n= f_n}} \dist^p_{(\g,{f}_n^*\h)}(d f_n, \SO{\g,{f}_n^*\h}) \, \dVolg + C\,\Volg\brk{\BRK{f'_n\ne f_n }} \\
&\quad \le \int_{\M} \dist^p_{(\g,{f}_n^*\h)}(d f_n, \SO{\g,{f}_n^*\h}) \, \dVolg + C\,\Volg\brk{\BRK{f'_n\ne f_n }}  \to 0.
\end{split}
\label{eq:Lpconv_of_dist_truncation}
\eeq
Moreover, for any $p<q<\infty$,
\beq
\begin{split}
&\int_\M \dist^q_{(\g,{f'}_n^*\h)}(d f'_n, \SO{\g,{f'}_n^*\h}) \, \dVolg \\
&\quad \le \int_\M (|df'_n|+c)^{q-p}\dist^p_{(\g,{f'}_n^*\h)}(d f'_n, \SO{\g,{f'}_n^*\h}) \, \dVolg \\
&\quad \le C\int_\M \dist^p_{(\g,{f'}_n^*\h)}(d f'_n, \SO{\g,{f'}_n^*\h}) \, \dVolg  \to 0.
\end{split}
\label{eq:Lqconv_of_dist_truncation}
\eeq
Next, we apply Steps I, II and III of the proof with $f_n$ replaced by $f'_n$ and any $q>d$.
We obtain that $f'_n$ converge in $W^{1,q}(\M;\N)$ to a smooth isometric immersion $f:\M\to \N$ (or equivalently $F'_n\to \iota \circ f$ in $W^{1,q}(\M;\R^D)$). 
By \eqref{eq:F_prime_n_minus_F_n_in_Sobolev}, it follows that $F_n\to F$ in $W^{1,p}(\M;\R^D)$, so by definition, $f_n\to f$ in $W^{1,p}(\M;\N)$.

%%%%%
\bigskip
\emph{Step V: Proof that under additional assumptions $f$ is an isometry}

Suppose that $f_n(\pl \M)\subset \pl \N$ and $\Volh\N = \Volg\M$. To show that $f$ is an isometry, it suffices to show that $f|_{\M^\circ}$ is a surjective isometry $\M^\circ\to\N^\circ$. Indeed, if this is the case, then, since $f$ is continuous and $\M$ is compact, $f(\M)$ contains $\N^\circ$ and is closed in $\N$, i.e., $f(\M) = \N$. Finally, $f$ is an isometry, because for every  $x,y\in\M$, let $\M^\circ\ni x_n\to x$ and $\M^\circ\ni y_n\to y$; by the continuity of the metrics $d_\M$ and $d_\N$,
\[
d_{\N}(f(x),f(y)) =
\lim_{n \to \infty} d_{\N}(f(x_n),f(y_n)) =
\lim_{n \to \infty} d_{\M}(x_n,y_n) =d_{\M}(x,y).
\]
Note that the intrinsic distance function on $\M^\circ$ is the same as the extrinsic distance $d_\M$, and similarly for $\N$, so there is no ambiguity here regarding which metric we use.

We proceed to show that $f|_{\M^\circ}$ is a Riemannian isometry $\M^\circ\to\N^\circ$. Recall that $f:\M \to \N$ is smooth as a map between manifolds with boundary, and $df \in\SO{\g,f^*\h}$ is invertible at every point. Thus for any interior point $x\in \M^\circ$, the image $f(x)$ must be an interior point of $\N$, hence $f(\M^\circ) \subset \N^\circ$. Since (by the inverse function theorem) $f:\M^\circ \to \N^\circ$ is a local diffeomorphism and in particular  an open map, $f(\M^\circ)$ is open in $\N^\circ$.

Since $f_n\to f$ in $W^{1,p}(\M;\N)$, it follows from the trace theorem (when viewing $f_n$ as elements in $W^{1,p}(\M;\R^D)$) that $f_n|_{\pl\M}\to f|_{\pl\M}$ in $L^p(\pl\M;\R^D)$, and (after taking a subsequence) pointwise almost everywhere in $\pl\M$. Since $f_n(\pl \M)\subset \pl \N$, and since $\pl\N$ is closed and $f$ is continuous we conclude that $f(\pl\M)\subset \pl \N$.
The reason for adopting an extrinsic viewpoint in the last argument is that the trace theorem relies upon the density of smooth functions in $W^{1,p}$. This density does not hold for mappings between manifolds for $p<d$. Using a truncation argument here would result in losing the condition that $f(\partial\M)\subset \partial\N$.

Let $f(x_n)\in \N^\circ$ converges to $y\in \N^\circ$. 
Since $\M$ is compact and $f$ is continuous, we may assume, by taking a subsequence,  that $x_n\to x\in \M$, and $y=f(x)$.
Since $f(\pl\M) \subset\pl\N$ and $y\in\N^\circ$, it follows that $x\in \M^\circ$, i.e., $y\in f(\M^\circ)$, which implies that $f(\M^\circ)$ is closed in $\N^\circ$.
We have thus obtained that $f(\M^\circ)$ is clopen in $\N^\circ$.
Since $\N^\circ$ is connected, $f(\M^\circ)=\N^\circ$, i.e., $f|_{\M^\circ}$ is surjective.

It remains to prove that $f|_{\M^\circ}$ is injective; this is where we use a volume argument.
The area formula for $f$ implies that
\beq
\label{eq:area_formula_manifolds_f}
\Volg \M = \int_\M |\Det df| \,\dVolg = \int_\N |f^{-1}(y)| \,d\Volh|_y \ge \Volh\N,
\eeq
where $|f^{-1}(y)|$ denotes the cardinality of the inverse image of $y$, and 
the last inequality follows from the surjectivity of $f|_{\M^\circ}$.
Since, by assumption, $\Volh\N = \Volg\M$, \eqref{eq:area_formula_manifolds_f} is in fact an equality, hence  
\[
\Volh \brk{\BRK{ q\in \N~:~ |f^{-1}(q)|>1 }} = 0.
\]
It follows that $f$ is injective on $\M^\circ$.
Indeed, assume $f(p_1) =q = f(p_2)$, where $p_1 \neq p_2\in \M^\circ$ and $q\in \N^\circ$.
Since $f$ is a local diffeomorphism, there exist disjoint open neighborhoods $U_i\ni p_i$ and $V\ni q$ such that $f(U_i) = V$, hence 
\[
\Volh\brk{ \{ q\in \N : |f^{-1}(q)|>1 \}} \ge \Volh(V) > 0, 
\]
which is a contradiction.
This completes the proof.

%%%%%

\bigskip
\emph{Step VI: If $f_n$ are diffeomorphisms then $f_n(\pl \M)\subset \pl \N$ and $\Volh\N = \Volg\M$}

If $f_n$ are diffeomorphisms, then obviously $f_n(\pl \M)\subset \pl \N$, and therefore \eqref{eq:area_formula_manifolds_f} holds.
It remains to show that $\Volh\N = \Volg\M$, and by \eqref{eq:area_formula_manifolds_f}, it is enough to show that $\Volg\M\le \Volh\N$.

For $p\ge d$, the equality of volumes is straightforward: 
since $\M$ is connected, $f_n$ are either globally orientation-preserving or globally orientation-reversing. 
Since $\dist(\GLm,\SO{d})=c(d)>0$, an orientation-reversing diffeomorphism $\phi:\M\to\N$ satisfies
\[
\int_\M\dist^p_{\g,\phi^*\h}(d\phi,\SO{\g,\phi^*\h})\,\dVolg \ge c^p \Volg\M.
\]
By \eqref{eq:Lpconv_of_dist}, $f_n$ are orientation-preserving for large enough $n$.
If $p\ge d$, then $\Volg\M=\Volh\N$ follows from Lemma~\ref{lem:bound_vol_distortion_dist_SO2} and \eqref{eq:Lpconv_of_dist}.

For $p<d$, we can use the truncated mappings $f_n'$ defined in step IV to show that $\Volg \M \le \Volg \N$.
By \eqref{eq:Lqconv_of_dist_truncation}, $\dist_{(\g,{f'}_n^*\h)}(d f'_n, \SO{\g,{f'}_n^*\h})\to 0$ in $L^q$ for any $q\in[1,\infty)$, but $f_n'$ are not diffeomorphisms, so we cannot use the above reasoning directly.
However, \eqref{eq:Lqconv_of_dist_truncation} (with $q=d$) and Lemma~\ref{lem:bound_vol_distortion_dist_SO1} imply that $|\Det df_n'|\to 1$ in $L^1(\M)$. 
Therefore,
\[
\begin{split}
\Volg \M  &= \int_\M |\Det df_n'| \,\dVolg + o(1)\\
	& = \int_{\{f_n=f_n'\}} |\Det df_n| \,\dVolg + \int_{\{f_n\ne f_n'\}} |\Det df_n'| \,\dVolg+ o(1)\\
	&\le \int_{\{f_n=f_n'\}} |\Det df_n| \,\dVolg + C\Volg(\{f_n\ne f_n'\}) + o(1) \\
	&\teq{\eqref{eq:F_n_ne_F_prime_n}} \int_{\{f_n=f_n'\}} |\Det df_n| \,\dVolg + o(1)\\
	&\le \int_{\M} |\Det df_n| \,\dVolg + o(1) = \Volh(\N) + o(1),
\end{split}
\]
where in the first inequality we used the fact that $f_n'$ are uniformly Lipschitz.
Therefore $\Volg \M \le \Volg \N$, and together with \eqref{eq:area_formula_manifolds_f} we obtain that $\Volg\M= \Volg\N$.
\hfill\ding{110}

%%%%%%%%%%%
\bigskip
We conclude this section with a number of remarks concerning the assumptions in Theorem~\ref{thm:Reshetnyak_M_N}.

\begin{enumerate}
\item Neither of the assumptions $f_n(\pl \M)\subset \pl \N$ and $\Volh\N = \Volg\M$, which were used to prove that $f$ is an isometry, can be dropped.  
Take for example $\M=[-1,1]^d$, and let $\N=\M/\sim$ be the flat $d$-torus with $\sim$ the standard equivalence relation.
Then $f_n:\M\to \N$ given by $f_n(x)=(1-1/n)x$ are injective and satisfy \eqref{eq:Lpconv_of_dist}, but converge uniformly to $\pi:\M\to \N$ the quotient map, which is obviously not an isometry but merely an isometric immersion.
This example shows that the assumption $f_n(\pl \M)\subset \pl \N$ cannot be relaxed.
	
In order to see that the condition $\Volg\M=\Volh\N$ cannot be relaxed, recall that there is an isometric immersion from the circle of radius $2$ in $\R^2$ into the circle of radius $1$.

\item Yet another alternative condition implying that $f$ is an isometry is the following "symmetric condition": 
there exist surjective mappings $f_n\in W^{1,p}(\M;\N)$ and $g_n\in W^{1,p}(\N;\M)$ such that 
\[
\dist_{(\g,f_n^*\h)}(d f_n, \SO{\g,f_n^*\h})\to 0 \qquad\text{in $L^p(\M)$},
\] 
and 
\[
\dist_{(\h,g_n^*\g)}(dg_n, \SO{\h,g_n^*\g})\to 0 \qquad\text{in $L^p(\N)$}.
\]
The proof follows the same steps as Theorem~\ref{thm:Reshetnyak_M_N} for both $f_n$ and $g_n$, resulting in $f_n\to f$, $g_n\to g$, where $f$ and $g$ are surjective isometric immersions. 
It follows that $g\circ f:\M\to \M$ is a surjective isometric immersion, and therefore (\cite[Theorem 1.6.15]{BBI01}) it is a metric isometry. 
Then, $f:\M\to \N$ is a metric isometry, and by the Myers-Steenrod theorem, it is a Riemannian isometry.
	
\item Generally, the compactness of $\N$ is essential for the proof of Theorem~\ref{thm:Reshetnyak_M_N}. 
However, we used the compactness of $\N$ only in the following places: 
(i) in Step I of the proof, where we applied the Poincar\'e inequality for the global mappings $F_n$; 
(ii) in Step I again, for the equivalence of intrinsic and extrinsic distances when we isometrically embed $\N\subset \R^D$;
and (iii) in Step V, for obtaining \eqref{eq:F_prime_n_minus_F_n_in_Sobolev}.
Thus, the compactness of $\N$ can be replaced by alternative assumptions, as long as these three properties hold.
In particular, the following holds:

\begin{corollary}
\label{cor:Reshetnyak_M_R_d}
Let $(\M,\g)$ be a compact $d$-dimensional manifold with $C^1$ boundary. Let $p>1$ and let $f_n\in W^{1,p}(\M;\R^d)$ be a sequence of mappings such that 
\[
\dist_{(\g,\euc)}(d f_n, \SO{\g,\euc})\to 0 \qquad\text{ in $L^p(\M)$},
\]
and $\int_\M f_n \dVolg = 0$. 
Then $f_n$ has a subsequence converging in $W^{1,p}(\M;\R^d)$ to a limit $f$, which is a smooth isometric immersion. In particular, $\M$ is flat.
\end{corollary}
In this case, the proof is in fact much simpler, since the global and local stages can be merged, and there is no need to locally replace $\h$ by $\euc$.
\end{enumerate}

%%%%%%%%%%%%%%%%%%%%%%%%%%%%%%%%%%%%%%%%%%%%%%%%%%%%%%%%%%%%%%
\section{Applications to convergence of manifolds}
\label{sec:applications}

The following definition is motivated by a series of works on the homogenization of manifolds with distributed singularities, and structural stability of non-Euclidean elasticity \cite{KM15,KM16,KM15b}:

%%%%%
\begin{definition}
\label{df:Lpq_convergence_of_manifolds}
Let $(\M_n,\g_n)_{n\in\bbN}$ and $(\M,\g)$ be compact $d$-dimensional Riemannian manifolds (possibly with $C^1$ boundary).
We say that the sequence $\M_n$ converges to $\M$ with exponents $p,q$ if there exists a sequence of diffeomorphisms $F_n:\M\to\M_n$ such that
\beq
\label{eq:dist_p_F_n_vanishes}
\|\dist_{(\g,F_n^*\g_n)}(d F_n, \SO{\g,F_n^*\g_n})\|_{L^p(\M,\g)}\to 0,
\eeq
\beq
\label{eq:dist_p_F_n_inverse_vanishes}
\|\dist_{(\g_n,(F_n^{-1})^*\g)}(d F_n^{-1}, \SO{\g_n,(F_n^{-1})^*\g})\|_{L^p(\M_n,\g_n)}\to 0,
\eeq
and the volume forms converge, that is
\beq
\label{eq:Det_F_n_inverse_converges}
\|\Det F_n -1\|_{L^q(\M,\g)}\to 0, \qquad \|\Det F_n^{-1} -1\|_{L^q(\M_n,\g_n)} \to 0.
\eeq
\end{definition}

%%%%%%%%%
\begin{theorem}
\label{thm:uniqueness_of_limit_Lpq_convergence}
The convergence in Definition~\ref{df:Lpq_convergence_of_manifolds} is well-defined for $q>1$ and $p\ge 2+ 1/(q-1)$: if $(\M_n,\g_n)\to (\M,\g)$ and $(\M_n,\g_n)\to (\N,\h)$, then $(\M,\g)$ and $(\N,\h)$ are isometric.
\end{theorem}
%%%%%%%%%

Note that if $p\ge d$, then \eqref{eq:dist_p_F_n_vanishes} and \eqref{eq:dist_p_F_n_inverse_vanishes} imply \eqref{eq:Det_F_n_inverse_converges} for $q=p/d$;
this follows from Lemma~\ref{lem:bound_vol_distortion_dist_SO1}.
Thus, the convergence in  Definition~\ref{df:Lpq_convergence_of_manifolds} is well-defined for $p\ge  d$ and $p\ge 2 + 1/(p/d-1)$, which after a short calculation amounts to  $p\ge \frac{1}{2}(d+2+\sqrt{d^2+4})$.

%%%%%%%%%
\begin{proof}
Assume that $\M_n\to \M$ with respect to $F_n:\M\to \M_n$, whereas $\M_n\to \N$ with respect to $G_n:\N\to \M_n$.
By the same argument as 
in the first comment below the proof of  Theorem~\ref{thm:Reshetnyak_M_N},
we may assume that both $F_n$ and $G_n$ are orientation-preserving for every $n$.

Eq.~\eqref{eq:Det_F_n_inverse_converges} for $F_n^{-1}$ implies that
\[
  \lim_n \Vol_{\g_n}\M_n = \lim_n \int_\M \Det F_n\,\Volg = \int_\M \Volg = \Volg\M.
\] 
By symmetry,
\beq
\label{eq:volume_equalities}
\Volg\M = \lim_n \Vol_{\g_n}\M_n= \Volh\N.
\eeq

Define the sequence of diffeomorphisms $H_n=G_n^{-1}\circ F_n:\M\to \N$.
We will show that
\[
\dist(dH_n,\SO{\g,H_n^*\h}) \to 0
\qquad \text{in $L^r(\M)$}
\]
for some $r\ge 1$. 
By Theorem~\ref{thm:Reshetnyak_M_N}, it follows that $\M$ and $\N$ are isometric.

Denote by $q_n\in \Gamma(\SO{\g,F_n^*\g_n}) \subset\Gamma(T^*\M\otimes F_n^*T\M_n)$ the section satisfying
\[
|dF_n - q_n| = \dist(dF_n,\SO{\g,F_n^*\g_n}),
\]
and by $r_n\in \Gamma(\SO{\g_n,(G_n^{_1})^*\h}) \subset\Gamma(T^*\M_n\otimes (G_n^{-1})^*T\N)$ the section satisfying
\[
|dG_n^{-1} - r_n| = \dist(dG_n^{-1},\SO{\g_n,(G_n^{-1})^*\h}).
\]
Then, since $dH_n = F_n^*dG_n^{-1}\circ dF_n$ and $F_n^*r_n\circ q_n \in\Gamma(T^*\M\otimes H_n^*T\N)$,
\beq
\begin{split}
\dist(dH_n,\SO{\g,H_n^*\h}) &\le |F_n^*dG_n^{-1}\circ dF_n  - F_n^*r_n\circ q_n|  \\
&\hspace{-2cm} = |(dF_n^*dG_n^{-1} - F_n^*r_n)\circ dF_n + F_n^*r_n\circ (dF_n - q_n)| \\
&\hspace{-2cm} \le |(dF_n^*dG_n^{-1} - F_n^*r_n)\circ dF_n| + |F_n^*r_n\circ (dF_n  - q_n)| \\
& \hspace{-2cm} \le F_n^*|dG_n^{-1} - r_n |\,  |dF_n| + |dF_n - q_n| \\
& \hspace{-2cm} = F_n^* \dist(dG_n^{-1},\SO{\g_n,(G_n^{-1})^*\h})\, |dF_n| + \dist(dF_n,\SO{\g,F_n^*\g_n}).
\end{split}
\label{eq:dist_Hn_SO_estimate}
\eeq
In the passage to the third line we used the triangle inequality; in the passage to the fourth line we used the fact that $F_n^*r_n$ is an isometry and the sub-multiplicativity of the Frobenius norm; in the passage to the fifth line we used the defining properties of $r_n$ and $q_n$.
 
By the definition of $p,q$-convergence, the second term on the right-hand side of \eqref{eq:dist_Hn_SO_estimate} tends to zero in $L^p(\M)$. Thus, it suffices to prove that
\beq
\label{eq:dist_Hn_SO_main_part}
F_n^*\dist(dG_n^{-1},\SO{\g_n,(G_n^{-1})^*\h})\,|dF_n|\to 0 \quad \text{in} \,\, L^r(\M)
\eeq
for some $1\le r < p$.

Since 
\[
|dF_n| \le \dist_{(\g,F_n^*\g_n)}(d F_n, \SO{\g,F_n^*\g_n}) + C(d),
\]
it follows that $\|dF_n\|_{L^p(\M,\g)}$ is uniformly bounded; the same holds for $\|dF_n^{-1}\|_{L^p(\M_n,\g_n)}$. Note that we used here the boundedness of $\Vol_{\g_n}\M_n$ in order to control the norm of $C(d)$ uniformly.

Using these observations along with H\"older's inequality, we obtain
\beq
\label{eq:very_long_Holder_calc}
\begin{split}
  & \| F_n^*\dist(dG_n^{-1},\SO{\g_n,(G_n^{-1})^* \h})\,|dF_n|\|_{L^r(\M,\g)} \\
&\quad \le \| F_n^*\dist(dG_n^{-1},\SO{\g_n,(G_n^{-1})^* \h})\|_{L^{rp/(p-r)}(\M,\g)} \,\| dF_n\|_{L^{p}(\M,\g)} \\
&\quad \le C\,\| F_n^*\dist(dG_n^{-1},\SO{\g_n,(G_n^{-1})^* \h})\|_{L^{rp/(p-r)}(\M,\g)} \\
&\quad = C\brk{\int_{\M_n}\dist^{pr/(p-r)}(dG_n^{-1},\SO{\g_n,(G_n^{-1})^* \h}) \,\frac{(F_n^{-1})^\star d\Vol_{\g}}{\dVoln} \,
\dVoln}^{(p-r)/rp} \\
&\quad \le C\left\| \dist(dG_n^{-1},\SO{\g_n,(G_n^{-1})^* \h})\right\|_{L^{qrp/(q-1)(p-r)}(\M_n,\g_n)} \,
			\left\|\frac{(F_n^{-1})^\star d\Vol_{\g}}{\dVoln} \right\|^{(p-r)/rp}_{L^{q}(\M_n,\g_n)} \\
&\quad \le  C\left\| \dist(dG_n^{-1},\SO{\g_n,(G_n^{-1})^* \h})\right\|_{L^{prq/(q-1)(p-r)}(\M_n,\g_n)}
\end{split}
\eeq
where $C>0$ is an appropriate constant varying from line to line.
Now choose $r= p/(2+1/(q-1))\ge 1$. 
Then $prq/(q-1)(p-r)\le p$, hence \eqref{eq:very_long_Holder_calc} reads
\[
\begin{split}
\| F_n^*\dist(dG_n^{-1},\SO{\g_n,\h})\,|dF_n|\|_{L^r(\M,\g)}
\le C\left\| \dist(dG_n^{-1},\SO{\g_n,\h})\right\|_{L^{p}(\M_n,\g_n)} \to 0.
\end{split}
\]
Therefore \eqref{eq:dist_Hn_SO_main_part} holds and the proof is complete.
\end{proof}
%%%%%%%%

\begin{remark}
Instead of \eqref{eq:Det_F_n_inverse_converges}, it is sufficient to assume that $\Vol_{\g_n}\M_n$ and $\|\Det F_n^{-1} -1\|_{L^q(\M_n,\g_n)}$ are bounded. 
Equation \eqref{eq:volume_equalities} in no longer valid, however \eqref{eq:Lpconv_of_dist} still holds for $H_n$ and some exponent $r\ge 1$,
so $\M$ and $\N$ are isometric by Theorem~\ref{thm:Reshetnyak_M_N}.
\end{remark}

\begin{example}
We now sketch two examples of convergence of manifolds according to Definition~\ref{df:Lpq_convergence_of_manifolds}.
In the first one, the limit coincides with the Gromov-Hausdorff limit; 
in the second the two limits are different.
Other examples, related to dislocation theory, can be found in \cite{KM15,KM15b}.

Note that these two examples involve singular metrics; in order to use the uniqueness result (Theorem~\ref{thm:uniqueness_of_limit_Lpq_convergence}) one needs to consider a smoothed version of the sequence. 
This can easily be done without changing the limit.
Note also that both examples are two-dimensional; this is for the sake of simplicity; both have higher dimensional generalizations.

\begin{enumerate}
\item Let $(\M,\g)$ be a two-dimensional Riemannian manifold.
For each $n$, choose a geodesic triangulation of $\M$, such that all the edge lengths are in $[1/n,3/2n]$.
In particular, the angles in all the triangles are bounded away from zero and $\pi$, uniformly in $n$. 
Denote the triangles by $(T_{n,i})_{i=1}^{k_n}$.
Construct $(\M_n,\g_n)$ by replacing each triangle $T_{n,i}$ with a Euclidean triangle $R_{n,i}$ having the same edge lengths.

Let $F_{n,i}:T_{n,i}\to R_{n,i}$ be a smooth diffeomorphism that preserves lengths along the edges of the triangles.
Since $T_{n,i}$ are very small, with angles bounded away from zero, they are "almost Euclidean", so $F_{n,i}$ can be chosen such that 
\[
\dist(dF_{n,i}, \SO{\g,F_{n,i}^*\g_n}),\,\, \dist(dF_{n,i}^{-1},\SO{\g_n,(F_{n,i}^{-1})^*\g}) < \frac{C}{n}
\]
for some $C>0$ independent of $n$.
$F_n:\M\to \M_n$ is then defined as the union of $F_{n,i}$.
The above bound on the distortion implies that $\M_n\to \M$ according to Definition~\ref{df:Lpq_convergence_of_manifolds} for every choice of exponents $p,q$ (including $p=q=\infty$).
$F_n$ are also maps of vanishing distortion in the metric-space sense, that is
\[
\max_{x,y\in \M} \left| d_\M(x,y) - d_{\M_n}(F_n(x),F_n(y)) \right| \to 0,
\]
and therefore $\M_n\to \M$ also in the Gromov-Hausdorff metric 
(see \cite{KM15b} for a similar construction). 

\item Let $\M=[0,1]^2$ endowed with the standard Euclidean metric.
Fix, say, $\e=1/10$. For every $n\in\bbN$,
define a discontinuous metric $\g_n$ on $\M$ as follows:
\[
\g_n(x,y) = \Cases{\e\euc &  (x,y)\in D \\
\euc & \text{otherwise},}
\]
where $(x,y)\in D$ if and only if there exists a $j\in\{1,\dots,n\}$ such that $|x - j/n| < 1/2n^2$ or $|y - j/n| < 1/2n^2$.
That is, $\g_n = \euc$ everywhere except on a set of $n$ horizontal and $n$ vertical strips of width $1/n^2$, in which it is shrunk isotropically by a factor $\e$. 
Let $F_n:(\M,\euc)\to (\M,\g_n)$ be the map $x\mapsto x$. 
Then $dF_n\in \SO{\euc,\g_n}$ everywhere except for a set of volume of order $1/n$; on that ``defective" set, $\dist(dF_n, \SO{\e,\g_n})$ is a constant independent of $n$.
The same properties apply for $F_n^{-1}$. It follows that $(\M,\g_n)\to (\M,\euc)$ according to Definition~\ref{df:Lpq_convergence_of_manifolds} for every choice of $p,q<\infty$.
On the other hand, $(\M,\g_n)$ converges in the Gromov-Hausdorff sense to the ``taxi-driver" $\ell^1$ metric on $[0,\e]^2$.

Note that there exist $\alpha>0$ and $p,q$ such that if we take $\e_n=n^{-\alpha}$ rather than a fixed $\e$, $(\M,\g_n)\to (\M,\euc)$ according to Definition~\ref{df:Lpq_convergence_of_manifolds}, whereas 
The Gromov-Hausdorff limit of this sequence is just the point.
\end{enumerate}
\end{example}

%%%%%%%%%%%%%%%%%%%%%%%%%%%%%%%%%%%%%%%%%%%%%%%%%%%
{\bfseries Acknowledgements}
We are grateful to Pavel Giterman, Amitai Yuval and Yael Karshon for useful discussions.
We also thank Deane Yang for suggesting the current form of \lemref{lem:Cofactor_grad_Determinant_bundle}.
This research was partially supported by the Israel Science Foundation (Grant No. 661/13), and by a grant from the Ministry of Science, Technology and Space, Israel and the Russian Foundation for Basic Research, the Russian Federation.

%%%%%%%%%%%%%%%%%%%%%%%%%%%%%%%%%%%%%%%%%%%%%%%%%%%
\appendix

\section{Sobolev spaces between manifolds}
\label{subsec:sobolev}

The following definitions and results are well-known; see \cite{Haj09, Weh04} for proofs and for further references.

Let $\M, \N$ be compact Riemannian manifolds, and let $D$ be large enough such that there exists an isometric embedding $\iota:\N\to \R^D$ (Nash's theorem).
For $p\in[1,\infty)$, we define the Sobolev space $W^{1,p}(\M;\N)$ by
\[
W^{1,p}(\M;\N) = \BRK{u:\M\to \N\,:\,\iota\circ u\in W^{1,p}(\M;\R^D)}. 
\]
This space inherits the strong and weak topologies of $W^{1,p}(\M;\R^D)$, which are independent of the embedding $\iota$.

Generally, these spaces are larger than the closure of  $C^\infty(\M;\N)$ in the strong/weak $W^{1,p}(\M;\R^D)$ topology.
However, when $p\ge d= \dim \M$, $W^{1,p}(\M;\N)$ is the strong closure of $C^\infty(\M;\N)$ in the strong topology \cite[Theorem 2.1]{Haj09}.

By the standard Sobolev embedding theorems, it follows that for $p>d$, $W^{1,p}(\M;\N)$ consists of continuous functions whose image is in $\N$ everywhere.
Moreover, $W^{1,p}(\M;\N)$ convergence implies uniform convergence for $p>d$.
Therefore, when $p>d$, $W^{1,p}(\M;\N)$ can be defined ``locally", namely
\[
\begin{split}
W^{1,p}(\M;\N) &= \Big\{u\in C(\M;\N) : \phi\circ u\in W^{1,p}(u^{-1}(U),\R^d), \\
&\qquad \,\,\text{for every local chart $\phi:U\subset\N\to \R^d$} \Big\}.
\end{split}
\]
In addition, $u_n\to u$ in $W^{1,p}(\M;\N)$ if and only if $u_n\to u$ uniformly, and 
$\phi\circ u_n \to \phi\circ u$ in $W^{1,p}(u^{-1}(U),\R^d)$ in every coordinate patch \cite[Lemmas B.5 and B.7]{Weh04}.

Moreover,  it follows from \cite[Theorem~4.9]{Hei05} that for $p>d$, $u\in W^{1,p}(\M;\N)$ is differentiable almost everywhere and that its strong and weak derivatives coincide almost everywhere. 
In particular, there is no ambiguity in Theorem~\ref{thm:weak_Piola_identity} and Theorem~\ref{thm:Liouville_M_N_Lipschitz}. 

Finally, note that for every $p\ge 1$ (including $p\le d$), there is a notion of weak derivative $du$ of $u\in W^{1,p}(\M;\N)$ (and not only of $\iota\circ u$), which is measurable as a function $T\M\to T\N$ \cite{CV16}. This implies, using local coordinates, that our energy density $x\mapsto \dist_{(\g_x,\h_{u(x)})} (du(x),\SO{\g_x,\h_{u(x)}})$ is indeed a measurable function for every $u\in W^{1,p}(\M;\N)$.

%%%%%%%%%%%%%%%%%%%
\section{Intrinsic determinant and cofactor}
\label{app:det_cof}

In this section we state and prove some useful properties of the determinant and cofactor operators defined in Section~\ref{sec:det_cof}.
Most of this section is linear algebra, which we include here for the sake of completeness.

%%%
\begin{proposition}
\label{pn:Det_and_volumes_1}
Let $V$ and $W$ be $d$-dimensional oriented inner-product spaces. Let $\star_V$ and $\star_W$ be their Hodge-dual operators. The inner-products and the orientations induce volume forms $\dVol_V$ and $\dVol_W$ (i.e., $\dVol_V(e_1,\dots,e_d)=1$ for every positively-oriented orthonormal basis of $V$).
Let $T\in\Hom(V,W)$. 
Then,
\[
\Det T = \frac{T^* \dVol_W}{\dVol_V}.
\]
\end{proposition}

The proof is immediate from the definitions, by choosing oriented orthonormal bases for $V$ and $W$. An immediate corollary is the following:
%%%%
\begin{corollary}
\label{pn:Det_and_volumes}
Let $f:\M\to \N$, then 
\[
\Det df = \frac{f^\star d\Volh}{\dVolg}
\]
\end{corollary}

%%%
\begin{proposition}
The following identity holds:
\label{pn:Cofactor_formula}
\[\Det A \, \id_V = A^T \circ \Cof A = (\Cof A)^T \circ A,\]
\end{proposition}

\begin{proof}
Let $v,u\in V$. Then,
\[
\begin{split}
(A^T\circ \Cof A(v),u)_V &= (\Cof A (v),Au)_W \\
&=  (-1)^{d-1} (\star_W^{d-1} \extp^{d-1}A \star_V^1 v,Au)_W  \\
& =  (-1)^{d-1}\star_W^d \brk{Au \wedge \star_W^1 \star_W^{d-1} \extp^{d-1}A \star_V^1 v} \\
&=  \star_W^d (Au \wedge \extp^{d-1}A \star_V^1 v)  \\
&= \star_W^d \brk{\extp^dA \brk{u \wedge \star_V^1 v}} \\
&= \star_W^d \brk{\extp^dA \star_V^0 \star_V^d \brk{u \wedge \star_V^1 v}} \\
&= \brk{ \star_W^d \extp^dA \star_V^0} \IP{u}{v}_V \\
&= \Det A \IP{u}{v}_V,
\end{split}
\]
where the passage to third line follows from the identity 
\[
\IP{v}{w}_{\Lam_p(V)} = \star_V^d(v \wedge \star_V^p w)
\] 
for $v,w \in \Lam_p(V)$.
Hence, for every $v\in V$,
\[
A^T\circ \text{Cof} A (v) = \Det A \,\id_{V}(v).
\]
\end{proof}

%%%
The following lemma is useful for proving the weak convergence of $\Cof df_n$ and $\Det df_n$:

\begin{lemma}
\label{lm:Cof_vs_cof}
Let $(V,\g)$ and $(W,\h)$ be $d$-dimensional oriented inner-product spaces.
Let $b=(b_1,\ldots,b_d)$ and $c=(c_1,\ldots,c_d)$ be arbitrary bases for $V$ and $W$.
Let $F\in\Hom(V,W)$, and let $A$ be its matrix representation in the given bases.
Denote by $A^T$, $\Cof A$ and $\Det A$ the matrix representations of $F^T$, $\Cof F$ and $\Det F$ in the given bases.
Denote by $A^t$, $\cof A$ and $\det A$ the transpose, cofactor and determinant of the \emph{matrix} $A$ (that is, the ``standard" linear-algebraic meaning of these notions).
Denote by $G$ and $H$ the matrix representations of $\g$ and $h$.
Then,
\beq
A^T = G^{-1} A^t H,  
\label{eq:ATformula}
\eeq
\beq
\Det A = \sqrt{\frac{\det H}{\det G}} \det A,
\label{eq:DetAformula}
\eeq
and
\beq
\Cof A = \sqrt{\frac{\det H}{\det G}} H^{-1} \cof A\, G.
\label{eq:CofAformula}
\eeq
\end{lemma}

\begin{proof}
Let $v\in V$ and $w\in W$. By definition $\h(Fv,w) = \g(v,F^Tw)$.
Moving to coordinates and writing this in matrix form, this reads
\[
v^t A^t H w = v^t G A^T w,
\]
from which \eqref{eq:ATformula} follows immediately.
Equation~\ref{eq:DetAformula} follows from Proposition \ref{pn:Det_and_volumes_1}. Using these two identities, \eqref{eq:CofAformula} follows from Proposition \ref{pn:Cofactor_formula} by a direct calculation.
\end{proof}

%%%
\begin{proposition}
\label{pn:abstract_weak_continuity_det_cof}

Let $\M,\N$ be $d$-dimensional oriented Riemannian manifolds, and let $\M$ be compact.
Let $f_n\in W^{1,p}(\M; \N)$ with $p>d$. 
If $f_n\weakly f$ in $W^{1,p}(\M; \N)$, then
\[
\Det df_n \weakly \Det df \quad\text{in $L^{p/d}(\M)$},
\]
and 
\[
\Cof df_n \weakly  \Cof df \quad\text{in $L^{p/(d-1)}(\M;\N)$},
\]
where the last convergence is understood locally, for every coordinate chart around $x\in \M$ and $f(x)\in \N$. 

\end{proposition}

\begin{proof}
The case $\M\subset \R^d$, $E=\R^d$ is a classical result in the theory of Sobolev mappings, see e.g. \cite[Section 8.2.4]{Eva98} (this reference only considers the determinant, however the same proof applies for the cofactor matrix).

Still in a compact Euclidean setting, if a sequence $g_n\in L^q(\W)$ weakly converges to $g$ in $L^{q}(\W)$ for some $q$, then $\phi g_n\weakly \phi g$ in $L^q(\W)$ for every smooth function $\phi$.
The proposition follows now from Lemma \ref{lm:Cof_vs_cof} by working in local coordinates and using the Euclidean result, since $H$, $G$, their inverses and determinants are all smooth functions of the coordinates, and $f_n\to f$ uniformly.
\end{proof}

%%%%%%%%%%%%%%%%%%%%%%%%%%%%%%%%%%%%%%%%%%%%%%%%%%%%%%%%%
\subsection{Derivative of the determinant: Proof of Lemma~\ref{lem:Cofactor_grad_Determinant_bundle}}
\label{app:deriv_det_bundle}

\lemref{lem:Cofactor_grad_Determinant_bundle} is concerned with the differentiation of the determinant of a bundle morphism between vector bundles. Since the intrinsic definition of the determinant (\ref{def:intrinsic_det}) involves the Hodge-dual, we first prove Lemma~\ref{lem:Hodge_star_cov_diff_commute} below regarding the behavior of the Hodge operator with respect to covariant differentiation. 

Let $\M$ be a smooth $d$-dimensional manifold.
Let $E$ be an oriented vector bundle over $\M$ (of arbitrary finite rank $n$), endowed with a Riemannian metric $\h$ and a metric affine connection $\nabla^E$. 
Note that
$\nabla^E$ induces a connection on $\Lambda_k(E)$ (also denoted by $\nabla^E$); this induced connection is compatible with the metric induced on $\Lambda_k(E)$ by $\h$.

%%%%%%%%%
\begin{lemma}[Hodge-dual commutes with covariant derivative]
\label{lem:Hodge_star_cov_diff_commute}
Let $(E,\M)$ be defined as above.  
Denote by $\star^k_E$ the fiber-wise Hodge-dual $\Lambda_k(E)\to \Lambda_{n-k}(E)$ (which is induced by the orientation on $E$ and $\h$).
Then, 
\[
\star^k_{E} (\nabla^E_X  \beta)=\nabla^E_X (\star^k_{E} \beta)
\] 
for every $\beta \in \Gamma(\Lambda_k(E))$ and $X \in \Gamma(T\M)$.
\end{lemma}
%%%%%%%%%

%\begin{proof}
\emph{Proof:} 
 We first show that it suffices to prove this lemma for $k=0$. That is, assume that for every $\xi\in C^\infty(\M)\simeq \Gamma(\Lambda_0(E))$ and $X \in \Gamma(T\M)$,
\beq
\label{app:eq:zerocase}
\star^0_{E} (\nabla^E_X(\xi)) = \nabla^E_X (\star^0_{E}(\xi)).
\eeq

Let $\alpha,\beta\in  \Gamma(\Lambda_k(E))$ and let $X \in \Gamma(T\M)$.
By the Leibniz rule for covariant differentiation and the definition of the Hodge-dual, 
\beq 
\label{app:eq:eq_comm_star_deriv_a}
\begin{split}
\nabla^E_X (\alpha \wedge \star^k_E \beta) &=
\nabla^E_X \alpha \wedge \star^k_E \beta + 
\alpha \wedge \nabla^E_X(\star^k_E \beta) \\
&=  \star^0_E (\nabla^E_X \alpha,\beta)_\h  +  \alpha \wedge \nabla^E_X(\star^k_E \beta).
\end{split}
\eeq
On the other hand,
\beq 
\label{app:eq:eq_comm_star_deriv_b}
\begin{split}
\nabla^E_X (\alpha \wedge \star^k_E \beta) 
&=\nabla^E_X (\star^0_E (\alpha,\beta)_\h) \\
&= \star^0_E \brk{\nabla^E_X  (\alpha,\beta)_\h} \\
&=\star^0_E \brk{(\nabla^E_X \alpha,\beta)_\h +  (\alpha, \nabla^E_X \beta)_\h} \\
& = \star^0_E  (\nabla^E_X \alpha,\beta)_\h +  \alpha\wedge \star^k_E (\nabla^E_X \beta),
\end{split}
\eeq
where the passage from the first to the second line uses \eqref{app:eq:zerocase} for $\xi = (\alpha,\beta)_\h$.
Equalities \eqref{app:eq:eq_comm_star_deriv_a} and \eqref{app:eq:eq_comm_star_deriv_b} imply that
\[
\alpha \wedge \nabla^E_X(\star^k_E \beta) =  \alpha\wedge \star^k_E (\nabla^E_X \beta).
\]
Since this holds for every $\alpha\in \Gamma(\Lambda_k(E))$, we conclude that
$\nabla^E_X(\star^k_E \beta) = \star^k_E (\nabla^E_X \beta)$.

Thus, we turn to prove \eqref{app:eq:zerocase}. Let $\beta \in C^\infty(\M)\simeq \Gamma(\Lambda_0(E))$, and note that
\beq
\label{app:eq:eq_comm_star_deriv_c}
\star^0_E (\nabla^E_X  \beta)= (\nabla^E_X  \beta) \,\star^0_E(1), 
\eeq
where $\star^0_E(1)$ is the positive unit $d$-dimensional multivector.  Likewise,
\beq 
\label{app:eq:eq_comm_star_deriv_d}
\nabla^E_X (\star^0_{E}(\beta)) = \nabla^E_X (\beta\, \star^0_{E}(1)) =
(\nabla^E_X \beta)\, \star^0_{E}(1) + \beta\, \nabla^E_X (\star^0_{E}(1)).
\eeq
Comparing \eqref{app:eq:eq_comm_star_deriv_c} and \eqref{app:eq:eq_comm_star_deriv_d}, we conclude that \eqref{app:eq:zerocase} holds for every $\beta$ if and only if
\beq
\nabla^E_X (\star^0_{E}(1)) = 0,
\label{app:eq:zerocasereduction}
\eeq
which is indeed the case, because $\star^0_{E}(1)$ is the unit $d$-dimensional multivector and $\nabla^E$ is consistent with the metric.
\hfill\ding{110}

\emph{Proof of Lemma~\ref{lem:Cofactor_grad_Determinant_bundle}:}
   Let $e_1,\dots,e_d$ be a positive orthonormal frame of $E$. 
  \[
    \Det(A)=   \star^d_{F} \circ \extp^d A \circ \star^0_{E}(1)=   \star^d_W \extp^d A \big( e_1 \wedge \dots \wedge e_d \big)= \star^d_{F} \big( Ae_1 \wedge \dots \wedge Ae_d \big)
  \]

  Using the Leibniz rule for the wedge product, we get
\beq
\label{eq:eq_deriv_det}
\begin{split}
  V\Det A &= V \star^d_{F}\big( Ae_1 \wedge \dots \wedge Ae_d \big) \stackrel{(1)}{=}   \star^d_{F}  \nabla_V \big( Ae_1 \wedge \dots \wedge Ae_d \big) \\
 	&=  \star^d_{F}  \sum_{i=1}^d  Ae_1 \wedge \dots \wedge \nabla_V (Ae_i) \wedge \dots \wedge Ae_d   \\
	&= \star^d_{F}  \sum_{i=1}^d  Ae_1 \wedge \dots \wedge (\nabla_V A)e_i \wedge \dots \wedge Ae_d  + \star^d_{F}  \sum_{i=1}^d  Ae_1 \wedge \dots \wedge  A(\nabla_{V}e_i) \wedge \dots \wedge Ae_d,
\end{split}
\eeq
Where equality $(1)$ follows from \lemref{lem:Hodge_star_cov_diff_commute}. (Here we used the metricity of the connection on $F$). 

Analyzing the second summand, we get
\[
\begin{split}
 \star^d_{F}  \sum_{i=1}^d \extp^d A(e_1 \wedge \dots \wedge \nabla_Ve_i \wedge \dots \wedge e_d)&= 
 \star^d_{F} \extp^d A( \sum_{i=1}^d  e_1 \wedge \dots \wedge \nabla_Ve_i \wedge \dots \wedge e_d) \\
 &=
 \star^d_{F} \extp^d A\big( \nabla_V  (e_1 \wedge \dots \wedge e_i \wedge \dots \wedge e_d) \big)=0,
 \end{split}
 \]
where in the last equality we used the metricity of the connection on $E$.

After eliminating the second summand,  \eqref{eq:eq_deriv_det} becomes
\[
\begin{split}
V\Det A &= \star^d_{F}  \sum_{i=1}^d  Ae_1 \wedge \dots \wedge (\nabla_V A)e_i \wedge \dots \wedge Ae_d \\
	& =   \sum_{i=1}^d \star^d_F (-1)^{i-1} \big( (\nabla_V A)e_i \wedge A e_1  \wedge \dots \wedge \widehat Ae_i \wedge \dots \wedge Ae_d \big) \\
	& = \sum_{i=1}^d \star^d_F (-1)^{i-1} \big( (\nabla_V A)e_i \wedge  (-1)^{d-1} \star^1_F \star^{d-1}_F ( A e_1 \wedge \dots \wedge \widehat Ae_i \wedge \dots \wedge Ae_d) \big)  \\
 	&  =  (-1)^{d-1} \sum_{i=1}^d (-1)^{i-1} \star^d_F  \big( (\nabla_V A)e_i \wedge  \star^1_F \star^{d-1}_F ( A e_1 \wedge \dots \wedge \widehat Ae_i \wedge \dots \wedge Ae_d) \big)  \\
	 &   = (-1)^{d-1} \sum_{i=1}^d (-1)^{i-1} \IP{(\nabla_V A)e_i}{ \star^{d-1}_F ( A e_1 \wedge \dots \wedge \widehat Ae_i \wedge \dots \wedge Ae_d)}_F  \\ 
	 &= (-1)^{d-1} \sum_{i=1}^d \IP{(\nabla_V A)e_i}{ \star^{d-1}_F \big( \extp^{d-1} A ( \star_V^1 e_i )\big)}_F\\
	 &=\sum_{i=1}^d \IP{(\nabla_V A)e_i}{  (-1)^{d-1}  \star^{d-1}_F  \extp^{d-1} A \star_V^1 e_i }_F\\
 	&= \sum_{i=1}^d \IP{(\nabla_V A)e_i}{  \Cof A(e_i) }_F 
	=\sum_{i=1}^d \IP{(\Cof A)^T(\nabla_V A)e_i}{  e_i }_E \\
	&=\tr\left( \Cof A^T \circ (\nabla_V A) \right)
	= \IP{\Cof A}{\nabla_V A}_{E,F}.
  \end{split}
  \]
\hfill\ding{110}

%%%%%%%%%%%%%%%%%%%%

\section{Volume distortion and $\dist\brk{\cdot,\SOd}$}
Let $A \in M_d$ be a linear transformation. $A$ maps the unit cube  into a body whose volume is $\det A$. We may therefore view $|\det A-1|$ as a measure of \textbf{volume distortion} of A's action. Intuitively, when $A$ is close to an (orientation-preserving) isometry, its volume distortion should be small. 
The following lemma is a quantitative formulation of this claim:

\begin{lemma}
 \label{lem:bound_vol_distortion_dist_SO1}
Let $A \in M_d$. Then 
\[
|\det A - 1| \le  \brk{\dist\brk{A,\SOd} +1}^d-1
\]
\end{lemma}

\begin{proof}
Let $\sigma_1\le \sigma_2\le \ldots\le \sigma_d$ be the singular values of $A$, and define $r_1 = \sgn(\det A)\, \sigma_1$, $r_i = \sigma_i$ for $i=2,\ldots,d$.
We then have $\det A= \Pi_{i=1}^d r_i$ and for every $1\le i\le d$,
\[
\dist\brk{A,\SOd} = \sqrt{\sum_{j=1}^d (r_j-1)^2} \ge |r_i-1|.
\]
We will show that 
\beq
\label{eq:ineq_vol_dist}
|\Pi_{i=1}^d r_i-1| \le  \Pi_{i=1}^d (|r_i - 1|+1) - 1,
\eeq
which will complete the proof since it will follow that
\[
|\det A - 1| \le  \Pi_{i=1}^d (|r_i - 1|+1) - 1 \le   \brk{\dist\brk{A,\SOd} +1}^d-1.
\]

We turn to prove \eqref{eq:ineq_vol_dist}. 
Bounding from above is trivial:
  \[
    \Pi_{i=1}^d r_i \le  \Pi_{i=1}^d (|r_i - 1|+1)  
  \]

  The less trivial part is bounding from below. We need to show:
  \[
    \Pi_{i=1}^d r_i-1 \ge -\brk{ \Pi_{i=1}^d (|r_i - 1|+1) - 1 } = 1- \Pi_{i=1}^d (|r_i - 1|+1)
  \]
  which is equivalent to:
  \beq
\label{eq:ineq_vol_dist2}
    2 \le  \Pi_{i=1}^d r_i +  \Pi_{i=1}^d (|r_i - 1|+1)
  \eeq

First, assume $A\in \GLp$.
Note that if $r_j\ge 1$ for some $j$, 
\[
\Pi_{i=1}^d r_i +  \Pi_{i=1}^d (|r_i - 1|+1) \ge \Pi_{i\ne j} r_i +  \Pi_{i\ne j} (|r_i - 1|+1).
\]
Therefore, it is enough to prove \eqref{eq:ineq_vol_dist2} under the assumption that $r_i\in(0,1)$ for all $i$,
that is, to prove that
\[
 f(r_1,\ldots,r_d) = \Pi_{i=1}^d r_i +  \Pi_{i=1}^d (2-r_i) \ge 2.
\]
Notice that the inequality holds on the boundary of $[0,1]^d$, and therefore it is enough to prove that $f$ has no local minima at $(0,1)^d$.
Indeed, if $r=(r_1,\dots,r_d) \in \partial([0,1]^d)$ then there exists some $i$ such that $r_i=0$ or $r_i=1$. If $r_i=1$ the inequality holds by induction on the dimension. If $r_i=0$, the inequality reduces to $ \Pi_{j \neq i} (2-r_j) \ge 1$ which holds by the assumption $r_i \in (0,1)$.
  
Differentiating in the interior $(0,1)^d$ we obtain
\[
\frac{\pl f}{\pl r_j} = \Pi_{i\ne j} r_i - \Pi_{i\ne j} (2-r_i) < 0,
\]
since $r_i\in(0,1)$ for every $i$.
Therefore there are no local minima at $(0,1)^d$, which completes the proof for $A\in \GLp$.

For $A\notin \GLp$, we need to prove \eqref{eq:ineq_vol_dist2}.
Note that in this case $r_1\le 0$, and therefore $|r_1-1|+1= 2 - r_1$.
We obtain that
\[
\begin{split}
\Pi_{i=1}^d (|r_i - 1|+1)+ \Pi_{i=1}^d r_i 
	&= 2  \Pi_{i=2}^d (|r_i - 1|+1) - r_1\brk{ \Pi_{i=2}^d (|r_i - 1|+1)- \Pi_{i=2}^d r_i} \\
	&\ge 2 - r_1\brk{ \Pi_{i=2}^d (|r_i - 1|+1)- \Pi_{i=2}^d r_i}
\end{split}
\]
Now, the term in the parentheses is non-negative and $-r_1\ge 0$, and therefore \eqref{eq:ineq_vol_dist2} holds.
\end{proof}

\begin{lemma}
  \label{lem:bound_vol_distortion_dist_SO2}
Let $f:(\M,\g) \to (\N,\h)$ be an orientation-preserving diffeomorphism between compact manifolds. Then
\[
  |\Vol_\h(\N) -  \Vol_\g(\M)|  \le \int_{\M} \Brk{\brk{\dist\brk{df,\SO{\g,f^*\h}} +1}^d-1} \,   \dVolg  
\]
\end{lemma}

\begin{proof} 
\[
\Volh\N = \int_{\N} d\Vol_{\h} = \int_{\M} f^\star(d\Vol_{\h}) = \int_{\M} (\Det df) \, \dVolg.
\]

Let $p\in\M$ and  let $v_i,w_i$ be positively oriented orthonormal bases for $T_p\M$ and $T_{f(p)}\N$. Let $A$ be the representing matrix of $df_p$ in these bases. Then, (i) $\det A>0$ since $f$ is orientation-preserving, (ii) $\Det df = \det A$ and (iii) 
\[
\dist_{(\g,f^*\h)}\brk{df,\SO{\g,f^*\h}} =\dist_{\euc}\brk{A,\SOd},
\]
where $\euc$ is the Euclidean metric.
Thus
\[ 
\begin{split}
|\Vol(\N) -  \Vol(\M)| &=  \left|\int_{\M} (\Det df-1) \, \dVolg \right| \\
&=  \left|\int_{\M} (\det A-1) \, \dVolg \right| \\
&\le \int_{\M} |\Det A-1| \dVolg \\
&\le \int_{\M} \Brk{\brk{\dist_{\euc}\brk{A,\SOd} +1}^d-1} \,   \dVolg \\
& =  \int_{\M} \Brk{\brk{\dist\brk{df,\SO{\g,f^*\h}} +1}^d-1} \,   \dVolg ,
\end{split}
\]
where the the passage to the fourth line follows from \lemref{lem:bound_vol_distortion_dist_SO1}.
\end{proof}

%%%%%%%%%%%%%%%%%%%%%%%%%%%%%%%%%%%%%%%%%%%%%%%%%%%%%%%%%%

\newcommand{\etalchar}[1]{$^{#1}$}
\providecommand{\bysame}{\leavevmode\hbox to3em{\hrulefill}\thinspace}
\providecommand{\MR}{\relax\ifhmode\unskip\space\fi MR }
\providecommand{\MRhref}[2]{%
  \href{http://www.ams.org/mathscinet-getitem?mr=#1}{#2}
}
\providecommand{\href}[2]{#2}

\bibliographystyle{amsalpha}

\end{document}